\documentclass{amsart}

\usepackage[T1]{fontenc} 
\usepackage{amsmath,amssymb,amsthm} 
\usepackage{graphicx} 
\usepackage{hyperref} 
\usepackage{enumitem} 
    \setlist[enumerate]{label=(\roman*)}
\usepackage{tikz} 
    \usetikzlibrary{cd}
\usepackage{array} 
\usepackage{hyphenat} 
\usepackage{microtype} 
\usepackage{hyphenat} 

\newcommand{\qq}{\mathbb{Q}} 
\newcommand{\nn}{\mathbb{N}} 
\newcommand{\zz}{\mathbb{Z}} 
\newcommand{\cc}{\mathbb{C}} 
\newcommand{\ff}{\mathbb{F}} 
\newcommand{\kk}{\mathbb{K}} 
\newcommand{\lang}[1]{\ensuremath{L(#1)}} 
\newcommand{\mqf}[1]{\ensuremath{Q_{#1}}} 
\newcommand{\ret}[1]{\ensuremath{\mathcal{R}_{#1}}} 
\newcommand{\retsubs}[2]{#1_{#2}} 
\newcommand{\mqu}[1]{\ensuremath{\mu^{#1}}} 
\newcommand{\isom}{\cong} 
\newcommand{\var}[1]{\ensuremath{\mathbf{#1}}} 
\newcommand{\pro}[1]{\ensuremath{\mathsf{Pro}(#1)}} 
\newcommand{\ab}[1]{\ensuremath{\var{Ab}_{#1}}} 
\newcommand{\freepro}{\widehat{F}} 
\newcommand{\free}{F} 
\newcommand{\given}{ : } 
\newcommand{\from}{\colon} 
\newcommand{\id}{id} 
\newcommand{\nil}{\mathrm{nil}} 
\newcommand{\emptyw}{\varepsilon} 
\newcommand{\mat}[1]{\ensuremath{M_{#1}}} 
\newcommand{\poly}[1]{\ensuremath{\chi_{#1}}} 
\newcommand{\rev}{\ast} 
\newcommand{\prop}{Proposition~}
\newcommand{\theo}{Theorem~}
\newcommand{\coro}{Corollary~}
\newcommand{\lem}{Lemma~}
\newcommand{\exa}{Example~}
\newcommand{\rem}{Remark~}

\DeclareMathOperator{\card}{Card} 
\DeclareMathOperator{\dee}{d} 
\DeclareMathOperator{\img}{Im} 
\DeclareMathOperator{\en}{End} 
\DeclareMathOperator{\mul}{mul} 
\DeclareMathOperator{\spl}{SL} 
\DeclareMathOperator{\gl}{GL} 
\DeclareMathOperator{\pdet}{pdet} 

\theoremstyle{plain}
\newtheorem{theorem}{Theorem}[section]
\newtheorem{lemma}[theorem]{Lemma}
\newtheorem{proposition}[theorem]{Proposition}
\newtheorem{corollary}[theorem]{Corollary}

\theoremstyle{definition}
\newtheorem{definition}[theorem]{Definition}

\theoremstyle{remark} 
\newtheorem{remark}[theorem]{Remark}

\newtheorem{example}[theorem]{Example}
\newtheorem{question}[theorem]{Question}

\begin{document}

\title{Pronilpotent quotients associated with primitive substitutions}

\author[H. Goulet-Ouellet]{Herman Goulet-Ouellet}

\address{University of Coimbra, CMUC, Department of Mathematics}

\email{hgouletouellet@student.uc.pt}



\begin{abstract}
    We describe the pronilpotent quotients of a class of projective profinite groups, that we call $\omega$-presented groups, defined using a special type of presentations. The pronilpotent quotients of an $\omega$-presented group are completely determined by a single polynomial, closely related with the characteristic polynomial of a matrix. We deduce that $\omega$-presented groups are either perfect or admit the $p$-adic integers as quotients for cofinitely many primes. We also find necessary conditions for absolute and relative freeness of $\omega$-presented groups. Our main motivation comes from semigroup theory: the maximal subgroups of free profinite monoids corresponding to primitive substitutions are $\omega$-presented (a theorem due to Almeida and Costa). We are able to show that the incidence matrix of a primitive substitution carries partial information on the pronilpotent quotients of the corresponding maximal subgroup. We apply this to deduce that the maximal subgroups corresponding to primitive aperiodic substitutions of constant length are not absolutely free.
\end{abstract}

\maketitle

\section{Introduction}
\label{s:intro}

In the early 2000s, Almeida established a connection between symbolic dynamics and free profinite monoids \cite{Almeida2003,Almeida2004,Almeida2007}. He showed that to each minimal shift space correponds a maximal subgroup of a free profinite monoid, later named the \emph{Schützenberger group} of the shift space. This group is obtained by taking the topological closure of the language of the shift space inside the corresponding free profinite monoid, and it defines an invariant of the shift space: two conjugate shift spaces have isomorphic Schützenberger groups \cite{Costa2006} (as do, even, flow equivalent shift spaces~\cite{Costa2021}).

In 2013, Almeida and Costa showed how to obtain presentations for the Schützenberger groups corresponding to substitutive minimal shift spaces using return substitutions and $\omega$-powers \cite{Almeida2013}. Using a similar process, every endomorphism of a free group of finite rank yields a presentation for some profinite group. Groups thus defined are called \emph{$\omega$-presented} and they are formally introduced in \S\ref{ss:omega}. The main goal of this paper is to describe the pronilpotent quotients of $\omega$-presented groups and apply this knowledge to study Schützenberger groups of primitive substitutions. In order to do this, we rely on several properties of maximal quotient functors which are presented in \S\ref{s:maxpronil}. Since $\omega$-presented groups are projective (\S\ref{ss:omega}), their maximal pronilpotent quotients are products of free pro-$p$ groups (\S\ref{ss:tate}). The ranks of these pro-$p$ components are completely determined, in a very straightforward way, by a single polynomial: the reciprocal of the characteristic polynomial of the incidence matrix of the free group endomorphism used in the $\omega$-presentation (\S\ref{ss:dimension}). In particular, for a given $\omega$-presented group, all the information about its pronilpotent quotients is contained in this single polynomial. 

Using all of this, we draw a number of conclusions. We show in \S\ref{ss:perfect} that these groups are either perfect, or have \emph{prime-rich} Abelianizations, in the sense that they admit the $p$-adic integers as quotients for cofinitely many primes. In \S\ref{ss:freeness}, we give necessary conditions for absolute and relative freeness of $\omega$-presented groups (on this topic, other results may also be found in a recent preprint by the author \cite{GouletOuellet2021}). This may be viewed as a contribution toward a solution to a problem proposed in 2013 by Almeida and Costa \cite[Problem~8.3]{Almeida2013}. 

In \S\ref{s:max}, we specialize these results to maximal subgroups of free profinite monoids corresponding to primitive substitutions. In this case, an $\omega$-presentation can be obtained using a return substitution \cite{Almeida2013}. Our first observation is that these groups are neither perfect nor pro-$p$, partially answering a question of Zalesskii reported by Almeida and Costa \cite{Almeida2013}. Extending an idea of Durand (\S\ref{ss:poly}), we show that the structure of the pronilpotent quotients of the maximal subgroup corresponding to a primitive aperiodic substitution is partially reflected in the characteristic polynomial of the substitution itself (\S\ref{ss:schutz}). The section culminates with one of our main results: the Schützenberger group of a primitive aperiodic substitution of constant length is not absolutely free (\theo\ref{t:uniform}). We conclude with a series of examples that illustrate various aspects of our results (\S\ref{ss:examples}). 

\section{Maximal pronilpotent quotients}
\label{s:maxpronil}

The aim of this section is to collect some general facts about maximal quotient functors, and more specifically about the pronilpotent one. We also recall along the way some definitions and set up some notation for the next sections. The first subsection is concerned with general properties of maximal quotient functors, while the second one focuses on the pronilpotent case.

\subsection{Maximal quotient functors}
\label{ss:def}

By a \emph{pseudovariety}, we mean a class of finite groups \var{H} closed under taking quotients and subgroups, and forming finite direct products. For the definition and basic properties of so-called \emph{pro-\var{H} groups}, the reader may wish to consult Ribes and Zalesskii's book on the topic \cite{Ribes2010a}. (Note that they use the term \emph{variety} instead of \emph{pseudovariety}.) Let 
\begin{itemize}
    \item \var{G} be the pseudovariety of all finite groups;
    \item $\var{G}_p$ be the pseudovariety of finite $p$-groups ($p$ a prime);
    \item $\var{G}_\nil$ be the pseudovariety of finite nilpotent groups.
\end{itemize}
Pro-\var{H} groups are respectively called \emph{profinite} when $\var{H}=\var{G}$, \emph{pro-$p$} when $\var{H}=\var{G}_p$ or \emph{pronilpotent} when $\var{H}=\var{G}_\nil$.

Given a profinite group $G$ and a pseudovariety \var{H}, we let $R_\var{H}(G)$ be the intersection of all clopen normal subgroups $N\trianglelefteq G$ such that $G/N\in \var{H}$. We further define $\mqf{\var{H}}(G) = G/R_\var{H}(G)$, and we denote by $\mqu{\var{H}}_G\from G\to\mqf{\var{H}}(G)$ the corresponding canonical epimorphism, $\mqu{\var{H}}_G(x)=xR_\var{H}(G)$. Note that $\mqf{\var{H}}(G)$ is pro-\var{H}: it is a subdirect product of the groups $G/N$, where $N$ ranges over all clopen normal subgroups $N\trianglelefteq G$ such that $G/N\in\var{H}$, and every subdirect product of pro-\var{H} groups is also pro-\var{H} \cite[\prop2.2.1(c)]{Ribes2010a}.

Let $\pro{\var{H}}$ be the category of pro-\var{H} groups equipped with continuous group homomorphisms, and consider the inclusion functor $I_\var{H}\from\pro{\var{H}}\to\pro{\var{G}}$. The next result is standard, although not usually stated in those terms. We include a proof for the reader's convenience. For more details on adjunctions, we refer to Mac~Lane's book \cite[{\S}IV]{MacLane1971}. The reader will also find there the definition of universal arrows used in the proof below \cite[{\S}III.1]{MacLane1971}.
\begin{proposition}[{cf. \cite[\lem3.4.1(a)]{Ribes2010a}}]\label{p:adjoint}
    For every pseudovariety $\var{H}$, $\mqf{\var{H}}$ is a functor which is a left adjoint of $I_\var{H}$. Moreover, \mqu{\var{H}} is a natural transformation which is the unit of this adjunction. 
\end{proposition}

\begin{proof}
    It suffices to show that for every profinite group $G$, the pair $(\mqf{\var{H}}(G),\mqu{\var{H}}_G)$ is a universal arrow from $G$ to $I_\var{H}$ \cite[{\S}IV.1, \theo2]{MacLane1971}.

    Let $H$ be a pro-\var{H} group and $\varphi\from G\to H$ be a continuous group homomorphism. The set $B$ of all clopen normal subgroups $N\trianglelefteq H$ such that $H/N\in\var{H}$ forms a neighborhood basis of the identity element of $H$ \cite[\theo2.1.3]{Ribes2010a}. Hence, $\ker(\varphi)$ is the intersection $\bigcap_{N\in B}\varphi^{-1}(N)$ and $G/\varphi^{-1}(N)\isom H/N\in \var{H}$. Thus, $R_\var{H}(G)\subseteq\ker(\varphi)$ and by standard properties of quotients, the map $\bar\varphi\from\mqf{\var{H}}(G)\to H$ defined by $\bar{\varphi}(xR_\var{H}(G)) = \varphi(x)$ is a well-defined morphism of profinite groups. In particular, it satisfies $\bar\varphi\mqu{\var{H}}_G=\varphi$, as required.
\end{proof}

Note that $\mqf{\var{H}}$ acts on morphisms as follows: if $\varphi\from G\to G'$ is a morphism of profinite groups, then $\mqf{\var{H}}(\varphi)$ is the unique morphism satisfying $\mqf{\var{H}}(\varphi)\mqu{\var{H}}_G = \mqu{\var{H}}_{G'}\varphi$. The group $\mqf{\var{H}}(G)$ is called the \emph{maximal pro-\var{H} quotient} of $G$, and when $\var{H}=\var{G}_\nil$ or $\var{G}_p$, the \emph{maximal pronilpotent quotient} or \emph{maximal pro-$p$ quotient} of $G$. Moreover, we abbreviate $\mqf{\var{G}_\nil}$ by $\mqf{\nil}$ and $\mqf{\var{G}_p}$ by $\mqf{p}$. 

Left adjoints are unique up to natural isomorphism \cite[{\S}IV.1, \coro1]{MacLane1971}. We make use of this fact to establish the next lemma. The proof uses a characterization of pro-\var{H} groups which already appeared in the previous proof: a profinite group $G$ is pro-\var{H} if and only if its identity element admits a neighborhood basis consisting of clopen normal subgroups $N\trianglelefteq G$ such that $G/N\in\var{H}$ \cite[\theo2.1.3]{Ribes2010a}.
\begin{lemma}\label{l:composition} 
    Let \var{H} and \var{K} be pseudovarieties. There is a natural isomorphism 
    \begin{equation*}
        \mqf{\var{H}}\mqf{\var{K}} \isom \mqf{\var{H\cap K}}.
    \end{equation*}
\end{lemma}

\begin{proof}
    Let $\var{L}=\var{H}\cap\var{K}$. We claim that a profinite group $G$ which is both pro-\var{H} and pro-\var{K} must also be pro-\var{L}. Let $B$ and $B'$ be neighborhood bases of the identity element of $G$ consisting of clopen normal subgroups $N\trianglelefteq G$ satisfying respectively $G/N\in\var{H}$ (for $N\in B$) and $G/N\in\var{K}$ (for $N\in B'$). Given $N\in B$, there is $N'\in B'$ such that $N'\subseteq N$, hence $G/N$ is a quotient of $G/N'$. In particular, $G/N\in\var{L}$ which proves the claim.

    By the previous paragraph, $\mqf{\var{H}}\mqf{\var{K}}$ is a functor $\pro{\var{G}}\to\pro{\var{L}}$. By the uniqueness of left adjoints, it suffices to show that $\mqf{\var{H}}\mqf{\var{K}}$ is a left adjoint of $I_{\var{L}}$, or equivalently that for every profinite group $G$, the pair $(\mqf{\var{H}}(K), \mqu{\var{H}}_K\mqu{\var{K}}_G)$, where $K=\mqf{\var{K}}(G)$, is a universal arrow from $G$ to $I_\var{L}$ \cite[{\S}IV.1, \theo2]{MacLane1971}.

    Let $\varphi\from G\to H$ be a morphism of profinite groups, where $H$ is pro-$\var{H}$. The universal properties of $\mqf{\var{K}}(G)$ and $\mqf{\var{H}}(K)$ give morphisms $\varphi'\from \mqf{\var{K}}(G)\to H$ and $\varphi''\from\mqf{\var{H}}(K)\to H$ such that $\varphi'\mqu{\var{K}}_G=\varphi$ and $\varphi''\mqu{\var{H}}_K=\varphi'$, as in the diagram below.
    \begin{equation*}
        \begin{tikzcd}
            G \rar{\mqu{\var{K}}_G} \drar[swap]{\varphi} & \mqf{\var{K}}(G) \dar[dashed]{\varphi'} \rar[equal] & K \rar{\mqu{\var{H}}_K} \dar[swap]{\varphi'} & \mqf{\var{H}}(K) \dlar[dashed]{\varphi''}\\
            & H \rar[equal] & H &
        \end{tikzcd}
    \end{equation*}
    Finally, we find that $\varphi''\mqu{\var{H}}_K\mqu{\var{K}}_G=\varphi'\mqu{\var{K}}_G=\varphi$, as required.
\end{proof}

Let us denote by $\freepro_\var{H}(X,\ast)$ the free pro-\var{H} group over a pointed Stone space $(X,\ast)$. Free pro-\var{H} groups have the universal property determined by the fact that $\freepro_\var{H}$ is the left adjoint of $U_\var{H}$, where $U_\var{H}$ is the forgetful functor from the category of pro-\var{H} groups to that of pointed Stone spaces (with the identity element of a group acting as basepoint). See \cite[{\S}3]{Ribes2010a} for more details. We abbreviate $\freepro_\var{G}$ by $\freepro$, $\freepro_{\var{G}_\nil}$ by $\freepro_\nil$, and $\freepro_{\var{G}_p}$ by $\freepro_p$ for every prime $p$. Groups of the form $\freepro(X,\ast)$, $\freepro_{p}(X,\ast)$ and $\freepro_\nil(X,\ast)$ are respectively called \emph{free profinite groups}, \emph{free pro-$p$ groups} and \emph{free pronilpotent groups}. Next is a slightly stronger version of a well-known result. 
\begin{lemma}\label{l:free}
    Let \var{H} and \var{K} be pseudovarieties. There is a natural isomorphism 
    \begin{equation*}
        \mqf{\var{H}}\freepro_\var{K}\isom\freepro_{\var{H}\cap\var{K}}.
    \end{equation*}
\end{lemma}

\begin{proof}
    The case $\var{K}=\var{G}$ is handled by \cite[\prop~3.4.2]{Ribes2010a}. To deduce the general case, use the first case together with \lem\ref{l:composition}, as follows:
    \begin{equation*}
        \mqf{\var{H}}\freepro_\var{K} \isom\mqf{\var{H}}\mqf{\var{K}}\freepro \isom \mqf{\var{H}\cap\var{K}}\freepro \isom \freepro_{\var{H}\cap\var{K}}.\qedhere
    \end{equation*}
\end{proof}

\subsection{Pronilpotent quotients of projective profinite groups}
\label{ss:tate}

Recall that a profinite group $G$ is projective when, for all profinite groups $H$ and $K$, and all morphisms of profinite groups $\varphi\from G\to H$ and $\psi\from K\to H$, with $\psi$ surjective, there exists a morphism $\varphi'\from K\to H$ such that $\psi\varphi'=\varphi$.
\begin{center}
    \begin{tikzcd}
        & G \dar{\varphi} \dlar[dashed,swap]{\varphi'} \\
        K\rar{\psi} & H
    \end{tikzcd}
\end{center}

Our main result for this section, \prop\ref{p:nilprojective} below, is a decomposition of the maximal pronilpotent quotient for projective profinite groups. Bearing in mind the properties of maximal quotient functors presented in \S\ref{ss:def}, it is a mostly straightforward consequence of Tate's characterization of projective pro-$p$ groups, which we now recall.
 
Let $A$ be a set (possibly infinite) equipped with its discrete topology, and \var{H} be a pseudovariety. Consider the pointed Alexandroff extension $(A\cup\{\ast\},\ast)$ of $A$. We stress that $A\cup\{\ast\}$ has an extra point even when $A$ is finite, so the term \emph{compactification} would be a misnomer. We write $\freepro_\var{H}(A)$ as a shorthand for $\freepro_\var{H}(A\cup\{\ast\},\ast)$, the free pro-\var{H} group over the pointed Alexandroff extension of $A$. Groups of the form $\free_\var{H}(A)$ are sometimes known as \emph{free pro-\var{H} groups on sets converging to 1}. Observe that if two sets $A$ and $B$ are in bijection, then $(A\cup\{\ast\},\ast)$ and $(B\cup\{\ast\},\ast)$ are homeomorphic. Hence, up to isomorphism, $\freepro_\var{H}(A)$ depends only on $\card(A)$. 

Let $G$ be a profinite group and $A$ be a set. Recall that a map $f\from A\to G$ \emph{converges to 1} when, for every clopen neighborhood $U$ of the identity element of $G$, the preimage $f^{-1}(U)$ contains all but finitely many elements of $A$. If $G$ is a pro-\var{H} group, then the pro-\var{H} group morphisms $\freepro_\var{H}(A)\to G$ are in bijection with the maps $f\from A\to G$ converging to 1. A result of Melnikov states that every free pro-\var{H} group is isomorphic to $\freepro_\var{H}(\mathfrak{m})$ for some cardinal $\mathfrak{m}$, called its \emph{rank} \cite[\prop3.5.12]{Ribes2010a}. In particular, every profinite group $G$ admits a map $A\to G$ converging to 1 for some set $A$, and we denote by $\dee(G)$ the smallest cardinality of such a set. Here is a statement for Tate's theorem extracted from the proof found in Fried and Jarden's book \cite[\prop22.7.6]{Fried2008}. 
\begin{theorem}[Tate]\label{t:tate}
    Let $G$ be a projective pro-$p$ group. Then, $G$ is isomorphic to $\freepro_p(\dee(G))$, the free pro-$p$ group of rank $\dee(G)$.
\end{theorem}

Let $G$ be a profinite group and $p$ be a prime. A \emph{$p$-Sylow subgroup} of $G$ is a closed pro-$p$ subgroup $H\leq G$ such that $[G:H]$ is coprime to $p$. (The definition of the index $[G:H]$ may be recalled in \cite[\S2.3]{Ribes2010a}). It is well known that $G$ is pronilpotent if and only if it has, for every prime $p$, a unique $p$-Sylow subgroup, which we denote $G_p$ \cite[\prop2.3.8]{Ribes2010a}. Moreover, in that case, $G=\prod_pG_p$ where $p$ ranges over all primes. In the next lemma, we record a simple observation which will prove useful in the sequel. Let us write $R_p$ in place of $R_{\var{G}_p}$ for every prime $p$ (so $R_p(G)$ denotes the intersection of the clopen normal subgroups $N\trianglelefteq G$ such that $G/N\in\var{G}_p$).
\begin{lemma}\label{l:sylow}
    Let $G$ be a pronilpotent group. For every prime $p$, the $p$-Sylow subgroup $G_p$ is isomorphic to $\mqf{p}(G)$. In particular, $G$ is isomorphic to $\prod_p\mqf{p}(G)$ where $p$ ranges over all primes. 
\end{lemma}

\begin{proof}
    Fix a prime $p$ and let $N$ be the kernel of the component projection $G\to G_p$. Since $G/N\isom G_p$ is pro-$p$, we have $R_p(G)\subseteq N$. Let $M\trianglelefteq G$ be a clopen normal subgroup such that $G/M$ is a finite $p$-group. Then, $N/(N\cap M)\isom (MN)/M$ is a subgroup of $G/M$, hence it is also a finite $p$-group. Note however that $N\isom G/G_p$, so the order of $N$ is coprime to $p$. Hence, $N/(N\cap M)$ is trivial and $N\subseteq M$. This shows that $N\subseteq R_p(G)$, finishing the proof.
\end{proof}

We now give our main result for this section. Let $\ab{p}$ be the pseudovariety of finite elementary Abelian $p$-groups. Given a profinite group $G$, we abbreviate $\dee(\mqf{\ab{p}}(G))$ by $\dee_p(G)$.
\begin{proposition}\label{p:nilprojective}
    For every projective profinite group $G$, we have
    \begin{equation*}
        \mqf{\nil}(G) \isom \prod_p\freepro_p(\dee_p(G)),
    \end{equation*}
    where $p$ ranges over all primes.
\end{proposition}

\begin{proof}
    By Lemmas~\ref{l:composition} and \ref{l:sylow}, we have that $\mqf{\nil}(G)$ is isomorphic to the product $\prod_p\mqf{p}(G)$ for $p$ ranging over all primes. Fix a prime $p$ and let $H$ stand for $\mqf{p}(G)$; it suffices to show that $H\isom\freepro_p(\dee_p(G))$. Since $G$ is projective, so is $H$ \cite[\prop22.4.8]{Fried2008}. By Tate's theorem (\theo\ref{t:tate}), it follows that $H\isom\freepro_p(\dee(H))$. On the one hand, we have $\dee(H) = \dee(\mqf{\ab{p}}(H))$ \cite[\lem22.7.4]{Fried2008}, while on the other hand, \lem\ref{l:composition} implies that $\mqf{\ab{p}}(H)\isom\mqf{\ab{p}}(G)$, hence $\dee(\mqf{\ab{p}}(H)) = \dee_p(G)$.
\end{proof} 

\begin{remark}\label{r:dim}
    A further consequence of \cite[\lem22.7.4]{Fried2008} is that $\mqf{\ab{p}}(G)$ and $(\zz/p\zz)^{\dee_p(G)}$ are isomorphic as elementary Abelian $p$-groups. If $G$ is finitely generated, then $\dee_p(G)$ is finite for every $p$ and it also gives the dimension of $\mqf{\ab{p}}(G)$ as a vector space over $\zz/p\zz$. This fails when $\dee_p(G)$ is infinite \cite[\rem22.7.5]{Fried2008}. 
\end{remark}

The decomposition of the maximal pronilpotent quotient above leads to the characterization of pronilpotent quotients below. Let us say that a profinite group $G$ is \emph{$\mathfrak{m}$-generated}, for a cardinal $\mathfrak{m}$, if there is a map $\mathfrak{m}\to G$ converging to 1 whose image generates a dense subgroup of $G$.
\begin{corollary}\label{c:description}
    Let $G$ be a projective profinite group and $H$ be a pronilpotent group. Then, $H$ is a continuous homomorphic image of $G$ if and only if for every prime $p$, the $p$-Sylow subgroup of $H$ is $\dee_p(G)$-generated.
\end{corollary}

\begin{proof}
    If $\psi\from G\to H$ is a surjective morphism of profinite groups, then so is $\mqf{p}(\psi)\from\mqf{p}(G)\to\mqf{p}(H)$, for every prime $p$. By the proof of the previous proposition, $\dee(\mqf{p}(G))=\dee_p(G)$, while by \lem\ref{l:sylow}, $\mqf{p}(H)$ is isomorphic to the $p$-Sylow subgroup $H_p$, hence $H_p$ is indeed $\dee_p(G)$-generated.

    On the other hand, assume that for every prime $p$, $H_p$ is $\dee_p(G)$-generated. Then, the proof of the previous proposition shows that $\mqf{p}(G)\isom\freepro(\dee_p(G))$, hence there is a surjective morphism of profinite groups $\psi_p\from\mqf{p}(G)\to H_p$. Since $H$ is the product of its Sylow subgroups, $\psi=\prod_p\psi_p$ gives a surjective morphism $\prod_p\mqf{p}(G)\to H$. The result follows since $\prod_p\mqf{p}(G)\isom\mqf{\nil}(G)$ is itself a continuous homomorphic image of $G$.
\end{proof}

\section{\texorpdfstring{$\omega$}{omega}-presented groups}
\label{s:omega}

In this section, we introduce $\omega$-presented groups (\S\ref{ss:omega}) and give a formula for the dimensions of the vector spaces $\mqf{\ab{p}}(G)$, where $G$ is an $\omega$-presented group (\S\ref{ss:dimension}). We then proceed to deduce a number of things about the structure of $\omega$-presented groups in \S\S\ref{ss:perfect} and \ref{ss:freeness}. 

\subsection{\texorpdfstring{$\omega$}{omega}-presentations}
\label{ss:omega}

Let $A$ be a set and $R$ be a subset of $\freepro(A)$. Denote by $N(R)$ the closed normal subgroup of $\freepro(A)$ generated by $R$. A \emph{presentation} of a profinite group $G$ is a pair $(A,R)$ with $A$ and $R$ as above and $\freepro(A)/N(R)\isom G$. We write $G \isom \langle A \mid R \rangle$. We call $A$ the set of \emph{generators} and $R$ the set of \emph{relators}. 

Projective profinite groups are also characterized by a special kind of presentation (\prop\ref{p:presentations}). This was first noticed by Lubotzky \cite[\prop1.1]{Lubotzky2001} and later extended by Almeida and Costa to the setting of profinite semigroups \cite[\prop2.4]{Almeida2013}. Both sources work with finitely generated objects, but for profinite groups, the characterization holds in full generality. We start with a lemma.
\begin{lemma}\label{l:presentations}
    Let $A$ be a set and $\psi$ be a continuous endomorphism of $\freepro(A)$. If $\psi$ is idempotent, then $\img(\psi) \isom \langle A \mid \psi(a)a^{-1} \given a\in A\rangle$.
\end{lemma}

\begin{proof}
    Letting $R=\{\psi(a)a^{-1} \given a\in A\}$, it is enough to show that $N(R)=\ker(\psi)$. It follows from the idempotence of $\psi$ that $\ker(\psi)=\{\psi(x)x^{-1} \given x\in\freepro(A)\}$, hence $N(R)\subseteq\ker(\psi)$. Showing that the remaining inclusion holds amounts to establishing that $H=\{x\in\freepro(A) \given \psi(x)x^{-1}\in N(R)\}$ is the whole of $\freepro(A)$. Equivalently, we have to show that $H$ is a closed subgroup of $\freepro(A)$ that contains $A$. That $H$ is closed follows readily from the fact that so is $N(R)$, together with the continuity of $\psi$ and basic properties of compact groups. That $H$ contains $A$ follows from its definition. Finally, for $x, y\in H$, we find that
    \begin{equation*}
        \psi(x^{-1}y)(x^{-1}y)^{-1} = \psi(x^{-1})\psi(y)y^{-1}x = x^{-1}(\psi(x)x^{-1})^{-1}(\psi(y)y^{-1})x,
    \end{equation*}
    and since $N(R)$ is a normal subgroup of $\freepro(A)$, we have $x^{-1}y\in H$.
\end{proof}

\begin{proposition}\label{p:presentations}
    Let $A$ be a set of cardinality $\mathfrak{m}$ and $G$ be an $\mathfrak{m}$-generated profinite group. Then, $G$ is projective if and only if $G\isom\langle A \mid \psi(a)a^{-1} \given a\in A\rangle$, where $\psi$ is a continuous idempotent endomorphism of $\freepro(A)$.
\end{proposition}

\begin{proof}
    Suppose that $G\isom\langle A \mid \psi(a)a^{-1} \given a\in A\rangle$ for some continuous idempotent endomorphism $\psi$ of $\freepro(A)$. By the previous lemma, this means that $G\isom\img(\psi)$, hence $G$ is isomorphic to a closed subgroup of $\freepro(A)$. Therefore, it must be projective \cite[\lem7.6.3]{Ribes2010a}. Conversely, assume that $G$ is projective. Since $G$ is $A$-generated, there is a surjective morphism of profinite groups $\alpha\from\freepro(A)\to G$. By projectivity, there is a morphism of profinite groups $\beta\from G\to\freepro(A)$ such that $\alpha\beta=\id_G$. Let $\psi$ be the composite $\beta\alpha$. Plainly, $\psi$ is an idempotent endomorphism and $\ker(\psi)=\ker(\alpha)$. Hence, $G=\img(\alpha)\isom\img(\psi)$ and the previous lemma concludes the proof.
\end{proof}   

We now restrict our attention to an even more specialized form of presentation. First, recall that if $G$ is a finitely generated profinite group, then $\en(G)$, the space of continuous endomorphisms of $G$ equipped with composition and the pointwise topology, is a profinite monoid \cite[\prop1]{Hunter1983}. In particular, for every endomorphism $\psi\in\en(G)$, the sequence $(\psi^n)_{n\geq 1}$ has a unique idempotent accumulation point given by $\psi^\omega=\lim_n\psi^{n!}$ \cite[\prop3.7.2 and~3.9.2]{Almeida2020a}. Given a finite set $A$, let $\free(A)$ denote the free group over $A$ and $\en(\free(A))$ be the set of endomorphisms of $\free(A)$. Viewing $\free(A)$ as a subgroup of $\freepro(A)$, it follows from the universal property of $\freepro(A)$ that every $\varphi\in\en(\free(A))$ admits a continuous extension $\widehat{\varphi}\in\en(\freepro(A))$.

\begin{definition}[$\omega$-presented groups]
    A profinite group $G$ is called \emph{$\omega$-presented} when it admits a presentation of the form $G \isom \langle A \mid \widehat{\varphi}^\omega(a)a^{-1} \given a\in A\rangle$, where $A$ is a finite set and $\varphi\in\en(\free(A))$. We then say that $\varphi$ \emph{defines an $\omega$-presentation of} $G$.
\end{definition} 
    
We emphasize that $\omega$-presented groups are finitely generated by definition. While it clearly follows from \prop\ref{p:presentations} above that every $\omega$-presented group is projective, it does not hold that every projective profinite group is $\omega$-presented. First and most obviously for not all projective profinite groups are finitely generated, as $\omega$-presented groups must be. But second and perhaps more interestingly, no $\omega$-presented group is a pro-$p$ group (\S\ref{ss:perfect}).

\subsection{Dimension formula}
\label{ss:dimension}

Following \S\ref{ss:tate}, the maximal pronilpotent quotient of a projective profinite group $G$ is completely determined by the cardinals $\dee_p(G)$, which in the finitely generated case agree, for each prime $p$, with the dimension of $\mqf{\ab{p}}(G)$ as a vector space over $\zz/p\zz$. \prop\ref{p:dimformula} below gives a simple formula for these dimensions in case $G$ is $\omega$-presented, which we call the \emph{dimension formula}. 

Before stating this proposition, we need to set up some notation. Let $\varphi$ be an endomorphism of $\free(A)$, where $A$ is a finite set. For every $a\in A$, let $|-|_a\from\free(A)\to\zz$ be the group homomorphism defined on $b\in A$ by $|b|_a = 1$ if $a=b$ and $|b|_a=0$ otherwise. The \emph{incidence matrix} of $\varphi$ is the $A\times A$ matrix over $\zz$ defined by
\begin{equation*}
    \mat{\varphi}(a,b) = |\varphi(b)|_a,\quad a,b \in A.
\end{equation*}
Given a prime $p$, we denote by \mat{p,\varphi} the matrix over $\zz/p\zz$ obtained by reducing modulo $p$ the coefficients of \mat{\varphi}. We define the characteristic polynomial of a square matrix $M$ by $\poly{}(x)=\det(x-M)$, with the convention that $\poly{}=1$ when $M$ is the empty matrix. We denote by \poly{\varphi} and \poly{p,\varphi} respectively the characteristic polynomial of \mat{\varphi} and \mat{p,\varphi}. Given a polynomial $\xi$ of degree $n$, we let $\xi^\rev$ be its \emph{reciprocal polynomial}, defined by $\xi^\rev(x) = x^n\xi(x^{-1})$.  We also call $\poly{\varphi}$ and $\poly{\varphi}^\rev$ the \emph{characteristic polynomial} and \emph{reciprocal characteristic polynomial} of $\varphi$. We record the following observations for future use.
\begin{remark}\label{r:pseudodet}
    Let $\kk$ be an algebraically closed field and $M$ be a square matrix over $\kk$. Let $\chi$ be the characteristic polynomial of $M$. Recall that $\chi$ splits over $\kk$ and that its roots are precisely the eigenvalues of $M$ in $\kk$. By Vieta's formulas, the degree of $\chi^\rev$ is the number of non-zero eigenvalues of $M$ counted with multiplicity. Moreover, up to a sign, the leading coefficient of $\chi^\rev$ is the product, taken with multiplicities, of the non-zero eigenvalues of $M$. This quantity is sometimes known as the \emph{pseudodeterminant} of $M$, and we denote it by $\pdet(M)$.
\end{remark}

\begin{proposition}[Dimension formula]\label{p:dimformula}
    Let $\varphi\in\en(\free(A))$ define an $\omega$\hyp{}presentation of a profinite group $G$. The dimension of $\mqf{\ab{p}}(G)$ over $\zz/p\zz$ is $\deg(\poly{p,\varphi}^\rev)$.
\end{proposition} 

\begin{proof}
    For convenience, we write $\ff_p=\zz/p\zz$. By \lem\ref{l:presentations}, we have $G\isom\img(\widehat{\varphi}^\omega)$, thus what we need is to compute the dimension of the image of $\mqf{\ab{p}}(\widehat{\varphi})^\omega$. Note that $\mqf{\ab{p}}(\widehat{\varphi})$ is a linear transformation of $\ff_p^A$ which may be identified with the matrix $\mat{p,\varphi}$. Moreover, $\en(\ff_p^A)$ is a finite monoid, so $\mat{p,\varphi}^\omega=\mat{p,\varphi}^n$ for infinitely many positive integers $n$, and it follows that
    \begin{equation*}
        \ker(\mat{p,\varphi}^\omega)=\{ x\in\ff_p^A \given \exists n\geq 1, \mat{p,\varphi}^n(x)=0\}.
    \end{equation*}
    But this is the generalized eigenspace of $\mat{p,\varphi}$ of eigenvalue 0, which has dimension $\mul_0(\poly{p,\varphi})$, the multiplicity of $0$ as a root of $\poly{p,\varphi}$ \cite[\coro7.5.3(2)]{Weintraub2019}. By the rank-nullity theorem,
    \begin{equation*}
        \dim(\img(\mat{p,\varphi}^\omega))=\deg(\poly{p,\varphi})-\mul_0(\poly{p,\varphi})=\deg(\poly{p,\varphi}^\rev).\qedhere
    \end{equation*}
\end{proof}

In light of \S\ref{ss:tate}, we then have the following.
\begin{theorem}\label{t:nilquotient}
    If $\varphi\in\en(\free(A))$ defines an $\omega$-presentation of a profinite group $G$, then $\mqf{\nil}(G) \isom \prod_p\freepro_p(\deg(\poly{p,\varphi}^\rev))$. Moreover, a pronilpotent group $H$ is a continuous homomorphic image of $G$ if and only if for every prime $p$, the $p$-Sylow subgroup of $H$ is $\deg(\poly{p,\varphi}^\rev)$-generated.
\end{theorem}

\begin{proof}
    By \prop\ref{p:nilprojective}, the first part follows if we show that $\dee_p(G)=\deg(\poly{p,\varphi}^\rev)$ for every prime $p$. Since $G$ is finitely generated, $\dee_p(G)$ is the dimension of $\mqf{\ab{p}}(G)$ over $\zz/p\zz$ (\rem\ref{r:dim}), which is indeed $\deg(\poly{p,\varphi}^\rev)$ by the dimension formula. The second part is proved in a similar way, using \coro\ref{c:description}.
\end{proof}

\subsection{Perfect \texorpdfstring{$\omega$}{omega}-presented groups}
\label{ss:perfect}

We now characterize perfect $\omega$-presented groups and describe what happens otherwise. We deduce that $\omega$-presented groups are never pro-$p$, and this includes the maximal subgroups of free profinite monoids defined by primitive substitutions (the topic of \S\ref{s:max}). The material in this section partially answers a question of Zalesskii reported in \cite[{\S}8]{Almeida2013}: can free pro-$p$ groups be realized as maximal subgroups of free profinite monoids? The answer is negative at least for the maximal subgroups corresponding to primitive substitutions. At time of writing, the question remains open for arbitrary minimal shift spaces.

Let us start with a characterization. By a \emph{perfect profinite group}, we mean a profinite group $G$ whose commutator subgroup is dense in $G$. Equivalently, the maximal pro-Abelian quotient of $G$ is trivial.
\begin{proposition}\label{p:perfect}
    Let $\varphi$ define an $\omega$-presentation of a profinite group $G$. Then, $G$ is perfect if and only if $\mat{\varphi}$ is nilpotent, i.e. $\mat{\varphi}^n=0$ for some $n\geq 1$.
\end{proposition}

\begin{proof}
    Note that non-trivial pronilpotent groups are not perfect (they are prosolvable), hence $G$ is perfect if and only if its maximal pronilpotent quotient is trivial. Then, by \theo\ref{t:nilquotient}, $G$ is perfect if and only if $\deg(\poly{p,\varphi}^\rev)=0$ for all primes. However, for cofinitely many primes, $\deg(\poly{p,\varphi}^\rev)=\deg(\poly{\varphi}^\rev)$. The latter is zero if and only if $\poly{\varphi}(x)=x^n$, and this is equivalent to $\mat{\varphi}$ being nilpotent by the Cayley--Hamilton theorem. 
\end{proof}

We deduce immediately the following result.
\begin{corollary}\label{c:alternative}
    If $G$ is $\omega$-presented, then either $G$ is a perfect profinite group, or the group $\zz_p$ of $p$-adic integers is a continuous homomorphic image of $G$ for cofinitely many primes $p$. In particular, non-trivial pro-$p$ groups are not $\omega$-presented.
\end{corollary}

\begin{proof}
    If $G$ is not perfect, then $\mat{\varphi}$ is not nilpotent and $\deg(\poly{\varphi}^\rev)>0$. As previously noted, $\deg(\poly{\varphi}^\rev)=\deg(\poly{p,\varphi}^\rev)$ for cofinitely many primes $p$. But by \theo\ref{t:nilquotient}, the product $\zz_p^{n}$ where $n=\deg(\poly{p,\varphi}^\rev)$ is the Abelianization of $\mqf{p}(G)$, hence it is a continuous homomorphic image of $G$. The last part follows by recalling that non-trivial prosolvable groups are not perfect.
\end{proof} 

Both alternatives occur in a non-trivial way. In \S\ref{ss:perfectex}, we exhibit a non-trivial perfect $\omega$-presented group. As for the other alternative, plenty of non-trivial examples are found among maximal subgroups of free profinite monoids (\S\ref{s:max}).
    
\subsection{Freeness}
\label{ss:freeness}

We give necessary conditions for absolute and relative freeness of $\omega$-presented groups. These results partially address \cite[Problem~8.3]{Almeida2013}. 

Let \var{H} be a pseudovariety. We say that a profinite group is \emph{free with respect to \var{H}} if it is isomorphic to $\freepro_\var{H}(A)$ for some set $A$. A profinite group is called \emph{relatively free} if it is free with respect to some pseudovariety \var{H}, and \emph{absolutely free} if moreover $\var{H}=\var{G}$, the pseudovariety of all finite groups. The next proposition characterizes relative freeness of maximal pronilpotent quotients of $\omega$-presented groups. Let $\pi$ be a set of primes. We let $\var{G}_{\nil,\pi}$ be the pseudovariety of finite nilpotent groups whose $p$-Sylow subgroups are trivial for all primes $p\notin\pi$.
\begin{proposition}\label{p:relfree}
    Let $\varphi\in\en(\free(A))$ define an $\omega$-presentation of a profinite group $G$. Let $\pi$ be the set of all primes $p$ such that $\deg(\poly{p,\varphi}^\rev)\neq 0$. Then, the following are equivalent.
    \begin{enumerate}
        \item\label{i:relfree} $\mqf{\nil}(G)$ is relatively free.
        \item\label{i:prime} For every prime $p$, $\deg(\poly{p,\varphi}^\rev)$ equals $\deg(\poly{\varphi}^\rev)$ or 0.
        \item\label{i:deg} $\mqf{\nil}(G)$ is free with respect to the pseudovariety $\var{G}_{\nil,\pi}$.
    \end{enumerate}
    In particular, $\mqf{\nil}(G)$ is a free pronilpotent group if and only if the pseudodeterminant of $\mat{\varphi}$ is $\pm1$.
\end{proposition}

\begin{proof}
    \emph{\ref{i:relfree} implies \ref{i:prime}.} Suppose that $\mqf{\nil}(G)$ is free with respect to a pseudovariety \var{H}, say $\mqf{\nil}(G)=\freepro_\var{H}(A)$ where $A$ is a finite set of cardinality $n$. The case $n=0$ is trivial: the only substitution on the empty alphabet has an empty incidence matrix, so then $\poly{\varphi}^\rev=1$ and \ref{i:prime} holds trivially. We may assume from now on that $n>0$. Fix a prime $p\in\pi$ and let $k=\deg(\poly{p,\varphi}^\rev)$. \theo\ref{t:nilquotient} implies that $\freepro_p(k)$ is a continuous homomorphic image of $G$, hence so is $(\zz/p\zz)^k$. Since $p\in\pi$, we have $k>0$, hence $\zz/p\zz\in\var{H}$ and $\ab{p}\subseteq\var{H}$. It follows from \lem\ref{l:free} that $\mqf{\ab{p}}(G)\isom(\zz/q\zz)^n$. In particular, $n$ is the dimension of $\mqf{\ab{p}}(G)$ over $\zz/p\zz$, so the dimension formula (\prop\ref{p:dimformula}) implies that $n=k$. But recall that $\deg(\poly{p,\varphi}^\rev)=\deg(\poly{\varphi}^\rev)$ for all sufficiently large $p\in\pi$, hence $n=\deg(\poly{\varphi}^\rev)$ and the result follows.

    \emph{\ref{i:prime} implies \ref{i:deg}.} Writing $d=\deg(\poly{\varphi}^\rev)$, we have by assumption $\deg(\poly{p,\varphi}^\rev)=0$ whenever $p\notin\pi$ and $\deg(\poly{p,\varphi}^\rev)=d$ otherwise. Applying \theo\ref{t:nilquotient} then gives
    \begin{equation*}
        \mqf{\nil}(G)\isom\prod_p\freepro_p(\deg(\poly{p,\varphi}^\rev))\isom\prod_{p\in \pi}\freepro_p(d),
    \end{equation*}
    which is indeed the free pro-$\var{G}_{\nil,\pi}$ group of rank $d$ (e.g. by Lemmas~\ref{l:free} and \ref{l:sylow}).

    That \ref{i:deg} implies \ref{i:relfree} is trivial, so it remains only to prove the last part of the statement. Recall that the leading coefficient of $\poly{\varphi}^\rev$ is equal, up to a sign, to the pseudodeterminant of $\mat{\varphi}$ (\rem\ref{r:pseudodet}). Thus, if $\pdet(\mat{\varphi})=\pm1$, then \ref{i:prime} is satisfied, $\pi$ is the set of all primes and by \ref{i:deg}, $\mqf{\nil}(G)$ is free pronilpotent. Conversely, suppose that $\mqf{\nil}(G)$ is free pronilpotent and that moreover there is a prime $p$ that divides $\pdet(\mat{\varphi})$. In particular, $\deg(\poly{p,\varphi}^\rev)<\deg(\poly{\varphi}^\rev)$ and $\mqf{\nil}(G)$ is non-trivial. Since $\mqf{\nil}(G)$ is relatively free, \ref{i:prime} must hold, thus $\deg(\poly{p,\varphi}^\ast)=0$. But then, \ref{i:deg} implies that the $p$-Sylow subgroup of $\mqf{\nil}(G)$ is trivial, contradicting the fact that $\mqf{\nil}(G)$ is a non-trivial free pronilpotent group. 
\end{proof}

We proceed to deduce necessary conditions for an $\omega$-presented group to be relatively or absolutely free. We think of these two results as quick tests for relative and absolute freeness. The second one extends a recent result of the author \cite[\coro4.7]{GouletOuellet2021}.
\begin{corollary}\label{c:reltest}
    Let $\varphi$ define an $\omega$-presentation of a profinite group $G$. If there is a prime $p$ such that $0<\deg(\poly{p,\varphi}^\rev)<\deg(\poly{\varphi}^\rev)$, then $G$ is not relatively free.
\end{corollary} 

\begin{proof}
    If $G$ is relatively free, then so is $\mqf{\nil}(G)$ by \lem\ref{l:free}. But the assumption that $0<\deg(\poly{p,\varphi}^\rev)<\deg(\poly{\varphi}^\rev)$ contradicts \ref{i:prime} from \prop\ref{p:relfree}.
\end{proof}

\begin{corollary}\label{c:abstest}
    Let $\varphi$ define an $\omega$-presentation of a profinite group $G$. If $\pdet(\mat{\varphi})$ is not $\pm1$, then $G$ is not absolutely free.
\end{corollary} 

\begin{proof}
    We prove the contrapositive. If $G$ is absolutely free, then it follows from \lem\ref{l:free} that $\mqf{\nil}(G)$ is a free pronilpotent group, hence $\pdet(\mat{\varphi})=\pm1$ by the last part of \prop\ref{p:relfree}.
\end{proof}

\subsection{A perfect example}
\label{ss:perfectex}

We conclude this section with an example of a perfect $\omega$-presented group. Consider the following endomorphism of the free group $\free(\{0,1\})$:
\begin{equation*}
    \psi\from\left\{\begin{array}{lll}
        0 & \mapsto & 010^{-1}1^{-1} \\
        1 & \mapsto & 0.
    \end{array}\right.
\end{equation*}
Let $P=\img(\widehat{\psi}^\omega)$ be the corresponding $\omega$-presented group. Plainly, $\mat{\psi}$ is nilpotent, so \prop\ref{p:perfect} ensures that $P$ is perfect. We now show that $P$ is non-trivial.

Consider a finite set $A$ and a finite group $H$. Let $\en(\freepro(A))$ act on the right of $H^A$ as follows: an element $t\in H^A$, viewed as a map $t\from A\to H$, naturally corresponds to a morphism of profinite groups $\widehat{t}\from\freepro(A)\to H$. For $\varphi\in\en(\freepro(A))$, define 
\begin{equation*}
    t^\varphi = (\widehat{t}\circ\varphi)|_A.
\end{equation*}
This gives a continuous right monoid action of $\en(\freepro(A))$ on $H^A$ \cite[\lem3.1]{Almeida2013}. Moreover, $H$ is a continuous homomorphic image of $\img(\varphi^\omega)$ if and only if there exists $t\in H^A$ and $k\geq 1$ such that $\{t(a) \given a\in A\}$ generates $H$ and $t^{\varphi^k}=t$ \cite[\prop3.2]{Almeida2013}. Let $\ff_4$ be the field with 4 elements.
\begin{proposition}
    The special linear group $\spl_2(\ff_4)$ is a continuous homomorphic image of $P$.
\end{proposition}

\begin{proof}
    Let $g$ be a generator of the multiplicative group $\ff_4^\times$, and consider the following $2\times 2$ matrices over $\ff_4$:
    \begin{equation*}
        u=\begin{pmatrix}1&1\\1&0\end{pmatrix},\quad v=\begin{pmatrix} 0&1\\1&g\end{pmatrix}.
    \end{equation*}
    One checks, via explicit computations, that
    \begin{equation*}
        (u,v)^{\psi^2}=(wuw^{-1}, wvw^{-1}), \text{ where } w=\begin{pmatrix} g & 1 \\ 0 & 1\end{pmatrix}.
    \end{equation*}
    It follows that $(u,v)^{\psi^{2k}} = (u,v)$, where $k$ is the order of the matrix $w$ in $\gl_2(\ff_4)$, the general linear group of dimension 2 over $\ff_4$. Let $H$ be the subgroup of $\spl_2(\ff_4)$ generated by $\{u,v\}$. By the aforementioned result \cite[\prop3.2]{Almeida2013}, it follows that $H$ is a continuous homomorphic image of $P$. As $P$ is perfect, so is $H$. But then $H$ is a non-trivial perfect subgroup of $\spl_2(\ff_4)$, and since the latter is the smallest non-trivial perfect group, we conclude that $H=\spl_2(\ff_4)$.
\end{proof}

\begin{question}\label{q:perfect}
    We wonder whether the above argument can be generalized to show that $\spl_2(\ff_{2^n})$ is a continuous homomorphic image of $P$ for every $n\geq 2$, where $\ff_{2^n}$ is the field with $2^n$ elements. 
\end{question}

Using GAP and SageMath \cite{GAP2020,Sage2020}, we were able to verify that the answer is positive for $2\leq n\leq 12$.

\section{Maximal subgroups of free profinite monoids}
\label{s:max}

We further study examples of $\omega$-presented groups arising from Almeida's correspondence between shift spaces and maximal subgroups of free profinite monoids (recalled in \S\ref{ss:correspondence}). We will focus on such maximal subgroups corresponding to primitive aperiodic substitutions. These groups are projective profinite groups by the main result of \cite{Rhodes2008}, and are in fact $\omega$-presented by the main result of \cite{Almeida2013}. This last result is of particular interest to us, so additional details are given in \S\ref{ss:return}. In \S\ref{ss:poly}, we revisit a result of Durand about eigenvalues of return substitutions. (By an eigenvalue of a substitution, we simply mean an eigenvalue of its incidence matrix.) In \S\ref{ss:schutz}, using this result, we relate more directly the characteristic polynomial of a primitive aperiodic substitution with the pronilpotent quotients of its Schützenberger group. We proceed to deduce specialized forms of the freeness tests of \S\ref{ss:freeness} and finally that the $\omega$-presented groups corresponding to primitive aperiodic substitutions of constant length cannot be free (\theo\ref{t:uniform}, our main result of this section). We finish, in \S\ref{ss:examples}, with a series of examples. 

\subsection{Almeida's correspondence}
\label{ss:correspondence}

We give a brief account of Almeida's correspondence, which associates to each minimal shift space a maximal subgroup in a free profinite monoid. For a more collected presentation of the topic, see Almeida et al.'s recent monograph \cite{Almeida2020a}. Given a finite discrete set $A$, consider the space $A^\zz$ equipped with the product topology. The map $\sigma\from A^\zz\to A^\zz$ defined by $\sigma(x)_n=x_{n+1}$ defines a self-homeomorphism of $A^\zz$ called the \emph{shift map}. A \emph{shift space} is a closed, non-empty subset $X\subseteq A^\zz$ satisfying $\sigma(X)=X$. Define the \emph{language of a shift space} $X$ to be the subset $\lang{X}$ of the free monoid $A^*$ formed by all words appearing as finite, contiguous subsequences in the elements $x\in X$. A shift space is called \emph{minimal} if it contains no shift space besides itself. It is well-known that a shift space is minimal if and only if \lang{X} is uniformly recurrent: this is essentially \cite[\prop5.2.3]{Almeida2020a}. (The definition of uniform recurrence may recalled e.g. in \cite[p.140]{Almeida2020a}.)

Almeida showed in \cite{Almeida2007} that if $X$ is a minimal shift space, then the topological closure of \lang{X} in the free profinite monoid $\widehat{A}^*$ intersects $\widehat{A}^*\setminus A^*$ in a regular $\mathcal{J}$-class. By standard semigroup theory, this $\mathcal{J}$-class contains maximal subgroups of $\widehat{A}^*$ that are (continuously) isomorphic to one another. We may think of these maximal subgroups as one single group, sometimes known as the \emph{Schützenberger group of $X$}. We say that a minimal shift space is \emph{periodic} if it is finite, or equivalently if its points have finite orbit under the shift map $\sigma$. Otherwise, by minimality, all points of $X$ have infinite orbit under $\sigma$ and we say that $X$ is \emph{aperiodic}. If $X$ is periodic, then its Schützenberger group is easily described: it is simply a free profinite group of rank 1 \cite[\theo7.5]{Almeida2006}. Hence, we restrict our attention to the aperiodic case.

Let now $\varphi$ be a primitive substitution over a finite alphabet $A$. That is, $\varphi$ is an endomorphism of $A^*$ whose incidence matrix $\mat{\varphi}$ is a primitive matrix. Equivalently, there is $n\in\nn$ such that, for all $a,b\in A$, the letter $b$ occurs in $\varphi^n(a)$. Such a substitution defines a shift space $X_\varphi\subseteq A^\zz$, whose language consists of all factors of the words $\varphi^n(a)$ for $n\geq 1$, $a\in A$ (see \cite[\S5.5]{Almeida2020a}). Going forward, we denote the language of $X_\varphi$ by $\lang{\varphi}$. Note that this language is uniformly recurrent (a proof may be found in \cite[\prop5.5.4]{Almeida2020a}), hence $X_\varphi$ is minimal. We say that $\varphi$ is \emph{aperiodic} if $X_\varphi$ is aperiodic in the above sense; otherwise, we say that $\varphi$ is \emph{periodic}. We denote the Schützenberger group of $X_\varphi$ by $G(\varphi)$ and by extension, we call it the \emph{Schützenberger group of $\varphi$}.

\subsection{Return substitutions}
\label{ss:return}

Return substitutions are one of the key tools for studying Schützenberger groups of primitive substitutions: they were used in \cite{Almeida2013} to obtain $\omega$-presentations for these groups. We give below the precise statement and reference for this result. But before, let us briefly recall what are return substitutions. Further details may be found in \cite{Durand1999}. 

Let $\varphi$ be a primitive substitution over a finite alphabet $A$. A pair of non-empty words $(u,v)$ is called a \emph{connection} of $\varphi$ when $uv\in\lang{\varphi}$ and there exists $n\geq 1$ such that $\varphi^n(u)\in A^*u$ and $\varphi^n(v)\in vA^*$. The least such $n$ is known as the \emph{order} of the connection. Consider the \emph{return set} $\ret{u,v}$, consisting of all words $r\in A^*$ such that $urv\in\lang{\varphi}$ and $urv$ starts and ends with consecutive occurrences of $uv$. Such a word is called a \emph{return word} to $(u,v)$. Recall that, by primitivity of $\varphi$, the language $\lang{\varphi}$ is uniformly recurrent. Hence, every long enough word in $\lang{\varphi}$ has an occurrence of $uv$, and the return set $\ret{u,v}$ must be finite. 

By uniform recurrence of $\lang{\varphi}$, there exists $l\in\nn$ such that for all $r\in\ret{u,v}$, $urv$ occurs in $u\varphi^{nl}(v)$, where $n$ is the order of $(u,v)$. We order $\ret{u,v}$ according to the leftmost occurrences of each $urv$ in $u\varphi^{nl}(v)$. Letting $A_{u,v}=\{0,\dots,\card(\ret{u,v})-1\}$, this ordering induces a monoid homomorphism $\theta_{u,v}\from A_{u,v}^*\to A^*$, which moreover does not depend on $l$. Note that $\ret{u,v}$ is the basis of a free submonoid of $A^*$ \cite[\lem17]{Durand1999}, so $\theta_{u,v}$ is injective. If $r\in\ret{u,v}$, then $u\varphi^n(r)v$ starts and ends with $uv$ and it follows that $\varphi^n(r)$ is uniquely a concatenation of elements of $\ret{u,v}$. In particular, we may define a substitution $\retsubs{\varphi}{u,v}$ of $A_{u,v}^*$ by $\varphi^n\theta_{u,v}=\theta_{u,v}\retsubs{\varphi}{u,v}$. We call $\retsubs{\varphi}{u,v}$ the \emph{return substitution} of $\varphi$ with respect to $(u,v)$. It is again a primitive substitution \cite[\lem21]{Durand1999}. 

We now recall a key result of Almeida and Costa implying that Schützenberger groups of primitive substitutions are $\omega$-presented. We stress that this is only valid for aperiodic substitutions. We also warn the reader that the original statement of the result is restricted to connections $(u,v)$ satisfying $|u|=|v|=1$, but this assumption is in fact never used in the proof. Relaxing this assumption can be convenient because it may happen that longer connections have less return words (e.g. the substitution $\varphi\from0\mapsto01,1\mapsto21,1\mapsto20$). 
\begin{theorem}[{\cite[\theo6.2]{Almeida2013}}]\label{t:return}
    Let $\varphi$ be a primitive aperiodic substitution and $(u,v)$ be a connection of $\varphi$. Then, $\retsubs{\varphi}{u,v}$, viewed as an endomorphism of $\free(A_{u,v})$, defines an $\omega$-presentation of $G(\varphi)$, that is $G(\varphi) \isom \langle A_{u,v} \mid \widehat{\retsubs{\varphi}{u,v}}^\omega(a)a^{-1} \given a\in A_{u,v}\rangle$.
\end{theorem}

Since all primitive substitutions have at least one connection (see \cite[\prop5.5.10]{Almeida2020a}), Schützenberger groups of primitive aperiodic substitutions are indeed $\omega$-presented. We further deduce the following. 
\begin{corollary}
    Schützenberger groups of primitive substitutions are neither perfect nor pro-$p$.
\end{corollary}

\begin{proof}
    Let $\varphi$ be a primitive substitution. If $\varphi$ is periodic, then the result is an easy consequence of \cite[\theo7.5]{Almeida2006}. From now on, we assume that $\varphi$ is aperiodic. By \theo\ref{t:return}, the group $G(\varphi)$ is $\omega$-presented, hence it cannot be pro-$p$ (\coro\ref{c:alternative}). Moreover, note that the $\omega$-presentation given by \theo\ref{t:return} is defined by an endomorphism with a primitive incidence matrix. Of course, primitive matrices are never nilpotent. In light of \prop\ref{p:perfect}, $G(\varphi)$ cannot be perfect. 
\end{proof}

\begin{remark}\label{r:proper}
    Let $\varphi\from A^*\to A^*$ be a primitive aperiodic substitution. It is called \emph{proper} when for some $a,b\in A$ and $k\geq 1$, $\varphi^k(c)\in aA^*\cap A^*b$ for all $c\in A$. In the proper case, there is a simpler version of the above theorem, which is sometimes more convenient: $\varphi$ itself defines an $\omega$-presentation of $G(\varphi)$ \cite[\theo6.4]{Almeida2013}.
\end{remark}

In any case, return substitutions may be effectively computed, for instance using the algorithm described in \cite[p.205]{Durand2012}, and as a result, the pronilpotent quotients of Schützenberger groups of primitive substitutions are quite transparent. Indeed, by Theorems~\ref{t:nilquotient} and \ref{t:return}, all is needed is a quick look at the reciprocal characteristic polynomial of any return substitution. However, computing return substitutions can be very tedious, as the example below shows. This motivates the results of \S\ref{ss:poly}.
\begin{example}\label{e:tedious1}
    Consider the following primitive substitution
    \begin{equation*}
        \varphi\from\left\{\begin{array}{lll}
            0 & \mapsto & 12\\
            1 & \mapsto & 22\\
            2 & \mapsto & 33\\
            3 & \mapsto & 00.
        \end{array}\right.
    \end{equation*}
    The pair of 1-letter words $(2,3)$ is a connection of $\varphi$ of order 12. The set $\ret{2,3}$ contains 12 return words with length ranging from 4 to 274. The return substitution $\retsubs{\varphi}{2,3}$ is thus defined on a 12-letter alphabet, and it is truly unyieldy: the images of the letters under $\retsubs{\varphi}{2,3}$ have lengths ranging from 821 to 97913. The other connections, which also have order 12, appear to give return substitutions that are comparable or even worse.
\end{example}

\subsection{Characteristic polynomials of return substitutions} 
\label{ss:poly}

Thankfully, we may relate, for a primitive substitution $\varphi$ with a connection $(u,v)$ of order $n$, the two polynomials $\poly{\varphi^n}$ and $\poly{\retsubs{\varphi}{u,v}}$, and in turn the reciprocal polynomials $\poly{\varphi^n}^\rev$ and $\poly{\retsubs{\varphi}{u,v}}^\rev$. In \cite[\prop9]{Durand1998a}, Durand shows that (up to taking a power) a primitive substitution shares the same eigenvalues as its one-sided return substitutions, except possibly for $0$ and roots of $1$. The main result of this subsection, \prop\ref{p:cyclo}, is a sharper version of this. We start with a technical lemma, also due to Durand. The lemma is outlined in the discussion preceding \cite[\prop9]{Durand1998a}. Since Durand's version of this lemma is stated for \emph{one-sided} return substitutions, we include a proof. 

Let $w, z$ be two words. An \emph{occurrence} of $z$ in $w$ is an integer $i\geq 0$ such that $w=xzy$ and $|x|=i$. The number of occurrences of $z$ in $w$ is denoted $|w|_z$. (Note that there is no conflict with the similar-looking notation introduced in \S\ref{ss:dimension}.) We also need to define incidence matrices for homomorphisms between free monoids over possibly different alphabets, which is done as follows. If $\varphi\from B^*\to A^*$ is a semigroup homomorphism where $A$ and $B$ are finite sets, then we let $\mat{\varphi}$ be the $A\times B$ matrix defined by
\begin{equation*}
    \mat{\varphi}(a,b) = |\varphi(b)|_a, \quad a\in A, b\in B.
\end{equation*}
Note that the formation of incidence matrices is compatible with composition, in the sense that $\mat{\varphi\psi}=\mat{\varphi}\mat{\psi}$ whenever $\varphi$ and $\psi$ are composable homomorphisms.

\begin{lemma}\label{l:durand}
    Let $\varphi\from A^*\to A^*$ be a primitive substitution and $(u,v)$ be a connection of $\varphi$ of order $n$. Then, there is a sequence of matrices $(K_l)_{l\in\nn}$ with integer coefficients making the following sets finite:
    \begin{equation*}
        \{\mat{\varphi^n}^l - \mat{\theta_{u,v}}K_l \given l\in\nn\},\quad \{\mat{\retsubs{\varphi}{u,v}}^l - K_l\mat{\theta_{u,v}} \given l\in\nn\}.
    \end{equation*}
\end{lemma}

\newcounter{step} 

\begin{proof}
    For simplicity, we replace $\varphi$ by $\varphi^n$ and assume that $n=1$. We define a map $f\from\lang{\varphi}\to\lang{\retsubs{\varphi}{u,v}}$ as follows. If $w\in\lang{\varphi}$ has an occurrence of $uv$, then it has a factorization $w=x_wy_wz_w$ satisfying (and defined by) the conditions
    \begin{equation*}
        \begin{cases}
            x_w\in A^*u, \\
            x_wy_w\in A^*u, \\
            |x_wv|_{uv}=1,
        \end{cases}
        \qquad
        \begin{cases}
            z_w\in vA^*,\\
            y_wz_w\in vA^*,\\
            |uz_w|_{uv}=1.
        \end{cases}
    \end{equation*}
    In this case, $y_w$ is a concatenation of elements of $\ret{u,v}$ and we let $f(w)=\theta_{u,v}^{-1}(y_w)$, which is well-defined by injectivity of $\theta_{u,v}$ \cite[\lem17]{Durand1999}. Otherwise, let $f(w)=\emptyw$, the empty word. The lemma is proved in 3 steps.

    \begin{figure}
        \begin{tikzpicture}[scale=.5]
            \draw (0,0) rectangle node {$x_w$} (2,1);
            \draw (2,0) rectangle node {$y_w$} (5,1);
            \draw (5,0) rectangle node {$z_w$} (7,1);
            \draw (1,1) rectangle node {$u$} (2,2);
            \draw (2,1) rectangle node {$v$} (3,2);
            \draw (4,1) rectangle node {$u$} (5,2);
            \draw (5,1) rectangle node {$v$} (6,2);
        \end{tikzpicture}
        \caption{Factorization of $w$ used to define the map $f$ in the proof of \lem\ref{l:durand}. The occurrences of $uv$ represented above are the first and last of $w$. They might overlap and even coincide.}
    \end{figure}
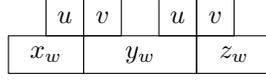

    \refstepcounter{step}\emph{Step \thestep.}\label{step:first} Let us write
    \begin{equation*}
        C_1 = \max\{|r| : r\in\ret{u,v}\},\quad C_2 = \min\{|r| : r\in\ret{u,v}\}.
    \end{equation*}
    We claim that, for every $w\in\lang{\varphi}$ and $b\in A$,
    \begin{equation}\label{e:ineq-thetaf}
        0\leq |w|_b-|\theta_{u,v} f(w)|_b \leq |w|-|\theta_{u,v} f(w)| \leq 2C_1+|uv|.
    \end{equation}
    If $w$ has no occurrence of $uv$, then the claim holds trivially. If $w$ has an occurrence of $uv$, then $|w|-|\theta_{u,v} f(w)|=|x_w|+|z_w|$ and the rightmost inequality of the claim follows from the upper bounds 
    \begin{equation}\label{e:ineq-xz}
        |x_w|\leq C_1+|u|;\quad |z_w|\leq C_1+|v|.
    \end{equation}
    To prove, say, the upper bound for $|x_w|$, note that $x_wv$ is a suffix of $urv$ for some $r\in\ret{u,v}$, hence $|x_w|\leq |ur| \leq C_1+|u|$. The upper bound for $|z_w|$ is obtained similarly. The two remaining inequalities of the claim are straightforward.

    \refstepcounter{step}\emph{Step \thestep.}\label{step:second} Let $s\in\lang{\retsubs{\varphi}{u,v}}$ and fix a factorization
    \begin{equation*}
        \theta_{u,v}(s) = w_1\dots w_k.
    \end{equation*}
    Assume further that each $w_i$ has at least one occurrence of $uv$, that is $|w_i|_{uv}\geq 1$. We claim that, for every letter $c\in A_{u,v}$, 
    \begin{equation}\label{e:ineq-ftheta}
        0\leq |s|_c - \sum_{i=1}^k|f(w_i)|_c\leq |s|-\sum_{i=1}^k|f(w_i)|\leq \frac{k(2C_1+|uv|)}{C_2}.
    \end{equation}
    By assumption, we have for every $i=1,\dots,k$ a factorization
    \begin{equation*}
        w_i = x_{w_i}y_{w_i}z_{w_i},
    \end{equation*}
    as described at the beginning of the proof. Since the cutting points of these factorizations correspond to occurrences of $uv$ in $\theta_{u,v}(s)$ shifted by $|u|$, there must be a corresponding factorization $s = \tilde{x}_1\tilde{y}_1\dots\tilde{y}_k\tilde{x}_{k+1}$ such that 
    \begin{equation*}
        \theta_{u,v}(\tilde{y}_i)=y_{w_i},\quad\theta_{u,v}(\tilde{x}_i)=
        \begin{cases}
            x_{w_1} & i=1,\\
            z_{w_{i-1}}x_{w_{i}} & 2\leq i\leq k,\\
            z_{w_k} & i=k+1.
        \end{cases}
    \end{equation*}
    (Use \cite[\prop2.6(2)]{Durand1998}.)  For $i=2,\dots,k$, we may use \eqref{e:ineq-xz} to conclude that 
    \begin{equation*}
        |\tilde{x}_{i}| \leq \frac{C_1+|u|}{C_2} + \frac{C_1+|v|}{C_2} \leq \frac{2C_1+|uv|}{C_2}.
    \end{equation*}
    Similarly, we have $|\tilde{x}_1|+|\tilde{x}_{k+1}|\leq\frac{2C_1+|uv|}{C_2}$. Noting that $\tilde{y}_i=f(w_i)$ by definition of $f$, we may now deduce the rightmost inequality of \eqref{e:ineq-ftheta},
    \begin{equation*}
        |s|-\sum_{i=1}^k|f(w_i)| = |\tilde{x}_1|+|\tilde{x}_{k+1}|+ \sum_{i=2}^k|\tilde{x}_{i}| \leq \frac{k(2C_1+|uv|)}{C_2}.
    \end{equation*}
    The two remaining inequalities are again straightforward.

    \refstepcounter{step}\emph{Step \thestep.}\label{step:third} For $l\in\nn$, define a homomorphism $\kappa_l\from A^*\to A_{u,v}^*$ by
    \begin{equation*}
        \kappa_l(a) = f(\varphi^l(a)),\quad a\in A.
    \end{equation*}
    Let $K_l=\mat{\kappa_l}$. We finish the proof by showing that the matrices $(K_l)_{l\in\nn}$ fulfill the requirements of the lemma. First, by \eqref{e:ineq-thetaf}, for all $a,b\in A$ and $l\in\nn$, we have
    \begin{equation*}
        0\leq |\varphi^l(a)|_b-|\theta_{u,v}\kappa_l(a)|_b = |\varphi^l(a)|_b-|\theta_{u,v}f(\varphi^l(a))|_b\leq 2C_1+|uv|.
    \end{equation*}
    Hence, the entries of the matrices $\mat{\varphi}^l-\mat{\theta_{u,v}}K_l$ can only take finitely many values, and this proves the first half of the statement. For the remaining half, we fix a letter $c\in A_{u,v}$ and we let $\theta_{u,v}(c)=a_1\dots a_k$, where $a_i\in A$. For every large enough $l$, the following factorization satisfies the condition of Step~\ref{step:second}:
    \begin{equation*}
        \theta_{u,v}(\retsubs{\varphi}{u,v}^l(c))=\varphi^l(a_1)\dots \varphi^l(a_k). 
    \end{equation*}
    Applying \eqref{e:ineq-ftheta} while noting that $k\leq C_1$ yields, for every letter $d\in A_{u,v}$,
    \begin{equation*}
        0\leq|\retsubs{\varphi}{u,v}^l(c)|_d-|\kappa_l\theta_{u,v}(c)|_d = |\retsubs{\varphi}{u,v}^l(c)|_d-\sum_{i=1}^k|f(\varphi^l(a_i))|_d \leq \frac{C_1(2C_1+|uv|)}{C_2}.
    \end{equation*}
    This shows that the entries of the matrices $\mat{\retsubs{\varphi}{u,v}}^l-K_l\mat{\theta_{u,v}}$ can take only finitely many values, completing the proof of the lemma.
\end{proof}

This leads us to the following result, which is our main result for the subsection. Roughly speaking, it states that, up to powers of $x$ and cyclotomic polynomials, a primitive substitution shares its characteristic polynomial with all of its return substitutions. As we already mentioned, this is a slightly sharpened version of a result of Durand \cite[\prop9]{Durand1998a}.
\begin{proposition}\label{p:cyclo}
    Let $\varphi\from A^*\to A^*$ be a primitive aperiodic substitution and $(u,v)$ be a connection of $\varphi$ of order $n$. Then, there exists a unique pair of coprime polynomials $\xi_1, \xi_2\in\zz[x]$ which are products of cyclotomic polynomials and satisfy
    \begin{equation*}
        \xi_1\poly{\varphi^n}^\rev = \pm\xi_2\poly{\retsubs{\varphi}{u,v}}^\rev. 
    \end{equation*}
\end{proposition}

\begin{proof}
    We note that cyclotomic polynomials, and hence their products, satisfy the relation $\xi^\rev=\pm\xi$. Hence, the result follows if we can show that for some positive integers $k_1$, $k_2$, we have
    \begin{equation}\label{eq:cyclo}
        x^{k_1}\xi_1(x)\poly{\varphi^n}(x) = x^{k_2}\xi_2(x)\poly{\retsubs{\varphi}{u,v}}(x),
    \end{equation}
    where $\xi_1, \xi_2\in\zz[x]$ are coprime and are both products of cyclotomic polynomials.

    As in the proof of the previous lemma, we may assume that $n=1$. Fix an eigenvalue $\lambda\in\cc$ of $\mat{\varphi}$ which is not 0 or a root of 1, and let $E(\lambda)$ and $E'(\lambda)$ denote the respective generalized eigenspaces of $\mat{\varphi}$ and $\mat{\retsubs{\varphi}{u,v}}$, that is
    \begin{gather*}
        E(\lambda) = \{ x\in\cc^A \given \exists k\geq 1, x(\mat{\varphi}-\lambda)^k=0\},\\
        E'(\lambda) = \{x\in\cc^{A_{u,v}} \given \exists k\geq 1, x(\mat{\retsubs{\varphi}{u,v}}-\lambda)^k=0\}.
    \end{gather*}
    Note that we view the elements of $\cc^A$ and $\cc^{A_{u,v}}$ as row vectors, so matrices act on the right. Fix an element $x\in E(\lambda)$, so $x\in\ker(\mat{\varphi}-\lambda)^k$ for some $k\geq 1$. Since $\varphi\theta_{u,v}=\theta_{u,v}\retsubs{\varphi}{u,v}$, we have $\mat{\varphi}\mat{\theta_{u,v}}=\mat{\theta_{u,v}}\mat{\retsubs{\varphi}{u,v}}$, and so
    \begin{equation*}
        x\mat{\theta_{u,v}}(\mat{\retsubs{\varphi}{u,v}}-\lambda)^k = x(\mat{\varphi}-\lambda)^k\mat{\theta_{u,v}} = 0.
    \end{equation*}
    Therefore, $x\mat{\theta_{u,v}}$ belongs to $E'(\lambda)$ and $\mat{\theta_{u,v}}$ gives a linear map $E(\lambda)\to E'(\lambda)$. We claim that the kernel of this map is trivial. Indeed, fix $x\in E(\lambda)\cap\ker(\mat{\theta_{u,v}})$, and let $k\in\nn$ be minimal such that $x(\mat{\varphi}-\lambda)^k=0$. Clearly, $k=0$ exactly when $x=0$. Thus, we assume $k>0$ and $x\neq 0$, and we argue by contradiction. Then, the vector $y=x(\mat{\varphi}-\lambda)^{k-1}$ is an eigenvector of $\mat{\varphi}$ of eigenvalue $\lambda$ which also belongs to $\ker(\mat{\theta_{u,v}})$. For every $l\geq 1$, let $Q_l=\mat{\varphi}^l-\mat{\theta_{u,v}}K_l$, where $K_l$ is the matrix from \lem\ref{l:durand}. It follows that 
    \begin{equation*}
        \lambda^ly = y\mat{\varphi}^l = y(\mat{\theta_{u,v}}K_l+Q_l) = yQ_l.
    \end{equation*}
    But \lem\ref{l:durand} states that the set of matrices $\{Q_l: l\in\nn\}$ is finite, so we may choose $1\leq l_1<l_2$ with $Q_{l_1}=Q_{l_2}$. Since $y\neq 0$, it follows that $\lambda^{l_1}=\lambda^{l_2}$, which contradicts the fact that $\lambda$ is not 0 or a root of 1. Thus, $x=0$ and $E(\lambda)$ is isomorphic to a subspace of $E'(\lambda)$. Using a similar argument, one proves that the left action of $\mat{\theta_{u,v}}$ on $\cc^{A_{u,v}}$, whose elements are now viewed as column vectors, induces an injective linear map $E'(\lambda)\to E(\lambda)$ (use instead $Q_l=\mat{\retsubs{\varphi}{u,v}}^l-K_l\mat{\theta_{u,v}}$). In particular, $\dim(E'(\lambda))=\dim(E(\lambda))$ for every $\lambda$ which is not 0 or a root of 1. 

    Next, recall that these dimensions give the algebraic multiplicities of $\lambda$ as a root of $\poly{\varphi}$ and $\poly{\retsubs{\varphi}{u,v}}$ respectively \cite[\coro7.5.3(2)]{Weintraub2019}. Hence, for some polynomials $\xi$, $\zeta_1$, $\zeta_2$ in $\cc[x]$ and some positive integers $k_1$, $k_2$, we have the following factorizations:
    \begin{equation*}
        \poly{\varphi}(x) = x^{k_2}\zeta_2(x)\xi(x),\quad \poly{\retsubs{\varphi}{u,v}}(x) = x^{k_1}\zeta_1(x)\xi(x),
    \end{equation*}
    where $\xi$ has no root equal to 0 or roots of 1, and all roots of $\zeta_1$, $\zeta_2$ are roots of 1. We claim that $\zeta_1$ and $\zeta_2$ are products of cyclotomic polynomials. Both cases being analogous, we argue only for $\zeta_2$. Choosing a root $\lambda$ of $\zeta_2$, we find that the minimal polynomial of $\lambda$ over $\qq$, say $\gamma$, must divide $\poly{\varphi}$. But $\gamma$ is a cyclotomic polynomial, thus its roots are all roots of 1. In particular, it follows that $\gamma$ is coprime with both $\xi$ and $x^{l_2}$. Hence, $\gamma$ must divide $\zeta_2$. Repeating this process until all roots of $\zeta_2$ are accounted for proves the claim.

    Let $\delta$ be the greatest common divisor of $\zeta_1$ and $\zeta_2$ in $\zz[x]$ and for $i=1,2$, let $\xi_i=\zeta_i/\delta$. Clearly we have $\xi_1\zeta_2=\xi_1\xi_2\delta=\xi_2\zeta_1$, so $\xi_1$ and $\xi_2$ together with the integers $k_1$ and $k_2$ satisfy \eqref{eq:cyclo}. That $\xi_1$ and $\xi_2$ are coprime and products of cyclotomic polynomials follows by construction. It remains to show that this is the only such pair. Suppose that $\xi'_1$ and $\xi'_2$ are products of cyclotomic polynomials satisfying \eqref{eq:cyclo} for some positive integers $l_1$, $l_2$. This readily implies $x^{l_1+k_2}\xi_1'\xi_2=x^{l_2+k_1}\xi_2'\xi_1$. Since $\xi_1$ and $\xi_2$ are coprime, we deduce that $\xi_1$ divides $\xi_1'$ and $\xi_2$ divides $\xi_2'$, thus proving uniqueness.
\end{proof}

\subsection{Pronilpotent quotients of Schützenberger groups}
\label{ss:schutz}

Let $\varphi$ be a primitive aperiodic substitution. Recall that \theo\ref{t:nilquotient} together with \theo\ref{t:return} imply that all the information concerning the pronilpotent quotients of $G(\varphi)$ is contained within the reciprocal characteristic polynomial of any return substitution of $\varphi$. The main result of \S\ref{ss:poly} means that the reciprocal characteristic polynomial of $\varphi$ itself carries at least partial information about the pronilpotent quotients of its Schützenberger group. This allows us to specialize some results from \S\ref{s:omega}, culminating with our main result (\theo\ref{t:uniform}), which states that Schützenberger groups of primitive aperiodic substitutions of constant length are never free. 

\begin{proposition}\label{p:dimcyclo}
    Let $\varphi$ be a primitive aperiodic substitution and $(u,v)$ be a connection of $\varphi$. Let $m_\varphi$ be the difference $\deg(\poly{\retsubs{\varphi}{u,v}}^\rev)-\deg(\poly{\varphi}^\rev)$. Then, for every prime $p$, we have $m_\varphi=\deg(\poly{p,\retsubs{\varphi}{u,v}}^\rev)-\deg(\poly{p,\varphi}^\rev)$. In particular, $\mqf{\ab{p}}(G(\varphi))$ has dimension $m_\varphi+\deg(\poly{p,\varphi}^\rev)$ over $\zz/p\zz$.
\end{proposition}

\begin{proof}
    Let $n$ be the order of the connection $(u,v)$ and let $\xi_1, \xi_2$ be the pair of polynomials given by \prop\ref{p:cyclo}. Fix a prime $p$, and for $i=1,2$, let $\xi_{p,i}$ be the polynomial obtained by reducing the coefficients of $\xi_i$ modulo $p$. It follows from \prop\ref{p:cyclo} that $\deg(\poly{p,\retsubs{\varphi}{u,v}}^\rev)-\deg(\poly{p,\varphi^n}^\rev) = \deg(\xi_{p,1})-\deg(\xi_{p,2})$. But observe that cyclotomic polynomials are monic, hence $\deg(\xi_{p,i})=\deg(\xi_i)$ for $i=1,2$. We claim that $\deg(\poly{p,\varphi^n}^\rev)=\deg(\poly{p,\varphi}^\rev)$. Indeed, let $\kk$ be the algebraic closure of $\zz/p\zz$ and view $\mat{p,\varphi}$ and $\mat{p,\varphi^n}$ as matrices over $\kk$. Then, the eigenvalues of $\mat{p,\varphi^n}$ are the $n$th powers of the eigenvalues of $\mat{p,\varphi}$ \cite[{\S}XIV, \theo3.10]{Lang2002}, hence the two matrices must have the same number of non-zero eigenvalues over $\kk$ counted with multiplicity. By \rem\ref{r:pseudodet}, their reciprocal characteristic polynomials must have the same degree, as claimed. Thus, for every prime $p$, we have 
    \begin{equation*}
        \deg(\poly{p,\retsubs{\varphi}{u,v}}^\rev)-\deg(\poly{p,\varphi}^\rev) = \deg(\xi_{1})-\deg(\xi_{2}).
    \end{equation*}
    But for $p$ large enough, the left-hand side of this equation is equal to $m_\varphi$, while the right-hand side is clearly independent of $p$. This completes the proof of the first part of the proposition. For the second part, recall that by \theo\ref{t:return}, $\retsubs{\varphi}{u,v}$ defines an $\omega$-presentation of $G(\varphi)$. Then, note that $m_\varphi+\deg(\poly{p,\varphi}^\rev)=\deg(\poly{p,\retsubs{\varphi}{u,v}}^\rev)$ and apply the dimension formula (\prop\ref{p:dimcyclo}). 
\end{proof}

We stress that $m_\varphi$ need not be positive (\exa\ref{e:negative}). We also observe that the integer $m_\varphi$ does not depend on the choice of $(u,v)$. Indeed, we found it to be equal, for every prime $p$, to $\dee_p(G(\varphi))-\deg(\poly{p,\varphi}^\rev)$, a quantity which clearly does not depend on $(u,v)$. This can rephrased as follows.
\begin{corollary}\label{c:deg}
    Let $\varphi$ be a primitive aperiodic substitution and $(u,v)$ be a connection of $\varphi$. Let $\xi_1, \xi_2$ be the two polynomials given by \prop\ref{p:cyclo}. The difference $\deg(\xi_1)-\deg(\xi_2)$ does not depend on the connection $(u,v)$.
\end{corollary}

\begin{remark}\label{r:cyclo}
    Let $(u,v)$ and $(u',v')$ be two connections of a primitive aperiodic substitution $\varphi$ sharing the same middle letters (i.e., $u$ and $u'$ share their last letter while $v$ and $v'$ share their first letter). Then, using a two-sided analog of \cite[\prop7]{Durand1998a}, one finds that $\poly{\retsubs{\varphi}{u,v}}^\rev = \poly{\retsubs{\varphi}{u',v'}}^\rev$. In particular, applying \prop\ref{p:cyclo} with either $(u,v)$ or $(u',v')$ yields the same pair $\xi_1,\xi_2$. This might not be true for connections that do not share the same middle letters (\exa\ref{e:cyclo}). 
\end{remark}

Because the value of $m_\varphi$ might be negative (\exa\ref{e:negative}), the relative freeness test of \coro\ref{c:reltest} cannot be applied directly using $\poly{\varphi}^\rev$ in place of $\poly{\retsubs{\varphi}{u,v}}^\rev$. However, we have the following weaker form.
\begin{proposition}\label{p:weaktest}
    Let $\varphi$ be a primitive aperiodic substitution. If there are two primes $p_1, p_2$ such that
    \begin{equation*}
        \deg(\poly{p_1,\varphi}^\rev)<\deg(\poly{p_2,\varphi}^\rev)<\deg(\poly{\varphi}^\rev),
    \end{equation*}
    then $G(\varphi)$ is not relatively free.
\end{proposition}

\begin{proof}
    Let $(u,v)$ be a connection of $\varphi$. By \prop\ref{p:dimcyclo}, adding $m_\varphi$ to each term in the inequality above yields
    \begin{equation*}
        0\leq \deg(\poly{p_1,\retsubs{\varphi}{u,v}}^\rev) < \deg(\poly{p_2,\retsubs{\varphi}{u,v}}^\rev) < \deg(\poly{\retsubs{\varphi}{u,v}}^\rev).
    \end{equation*}
    Since $\varphi_{u,v}$ defines an $\omega$-presentation of $G(\varphi)$ (\theo\ref{t:return}), we may apply \coro\ref{c:reltest} with $p_2$ to conclude that $G(\varphi)$ is not relatively free.
\end{proof}

\exa\ref{e:weaktest} gives an example where the test above is conclusive. For the absolute freeness test of \coro\ref{c:abstest}, the situation is more straightforward.
\begin{proposition}\label{p:subsfreetest}
    Let $\varphi$ be a primitive aperiodic substitution. If $\pdet(\mat{\varphi})$ is not $\pm1$, then the Schützenberger group $G(\varphi)$ is not a free profinite group.
\end{proposition}

\begin{proof}
    Let $(u,v)$ be a connection of $\varphi$ of order $n$. It follows from \prop\ref{p:cyclo} that $\pdet(\mat{\varphi^n}) = \pm\pdet(\mat{\retsubs{\varphi}{u,v}})$. Moreover, we observe that $\pdet(\mat{\varphi^n})=\pdet(\mat{\varphi})^n$. Hence, $\pdet(\mat{\varphi})$ equals $\pm1$ if and only if so does $\pdet(\mat{\retsubs{\varphi}{u,v}})$. Recall (\theo\ref{t:return}) that $\retsubs{\varphi}{u,v}$ defines an $\omega$-presentation of $G(\varphi)$ and apply \coro\ref{c:abstest} to conclude that $G(\varphi)$ is not absolutely free. 
\end{proof}

A substitution $\varphi\from A^*\to A^*$ is said to have \emph{constant length} when there is some integer $k\in\nn$ such that $|\varphi(a)|=k$, for all $a\in A$. Constant length primitive substitutions include the famous Thue--Morse substitution (\exa\ref{e:morse}), which was shown to have a non-free Schützenberger group in \cite{Almeida2013}. The next theorem generalizes this result.
\begin{theorem}\label{t:uniform}
    Let $\varphi$ be a primitive aperiodic substitution of constant length. Then $G(\varphi)$ is not absolutely free. 
\end{theorem}

\begin{proof}
    Assume that $k=|\varphi(a)|$ for every letter $a$. In light of \prop\ref{p:subsfreetest}, it is enough to show that the leading coefficient of $\poly{\varphi}^\rev$ is divisible by $k$. Note that the vector $(1,\dots,1)$ is a (left) eigenvector of $\mat{\varphi}$ of eigenvalue $k$ with coefficients in $\zz$, hence there is a factorization in $\zz[x]$
    \begin{equation*}
        \poly{\varphi}(x) = (x-k)\gamma(x).
    \end{equation*}
    Plainly then, $k$ divides the leading coefficient of $\poly{\varphi}^\rev$.
\end{proof}

\begin{remark}
    We may contrast the last result with the case of unimodular substitutions. Recall that a substitution is called \emph{unimodular} when its incidence matrix is invertible over $\zz$, or equivalently when its determinant is $\pm1$. If $\varphi$ is primitive, unimodular and aperiodic, then it follows from Propositions~\ref{p:relfree} and~\ref{p:cyclo} that $\mqf{\nil}(G(\varphi))$ is free pronilpotent. Therefore, $G(\varphi)$ is indistiguishable from a free profinite group in its finite nilpotent quotients. The same argument applies for all primitive aperiodic substitutions whose incidence matrix has pseudodeterminant~$\pm1$.

    However, while unimodularity of $\varphi$ guarantees that $\mqf{\nil}(G(\varphi))$ is free pronilpotent, it does not, by any means, guarantee that $G(\varphi)$ itself is free, even relatively so. The reader can find in \cite[\S6]{GouletOuellet2021} an example of a primitive substitution on 4 letters that induces an automorphism of the free group (hence is unimodular), but whose Schützenberger group is not relatively free.
\end{remark}

In a recent paper \cite{Costa2021}, Costa and Steinberg proved that the Schützenberger groups (and, thus, their maximal pronilpotent quotients) of irreducible shift spaces are invariant under flow equivalence. In particular, this means that our results provide flow invariants for shift spaces of primitive aperiodic substitutions. Here are some low hanging fruits. Among shift spaces defined by primitive aperiodic substitutions, we found the following to be invariant under flow equivalence:
\begin{enumerate}
    \item\label{i:seq} the sequence of integers $(\deg(\poly{p,\retsubs{\varphi}{u,v}}^\rev))_p$ indexed by prime numbers;
    \item\label{i:pdet} the set of primes dividing $\pdet(\varphi)$.
\end{enumerate}
These are reasonably easy to compute, especially the second one, but they are also fairly weak. For instance, \ref{i:pdet} cannot distinguish primitive unimodular substitutions that are defined on the same alphabet. At least, these invariants suffice to separate, for instance, unimodular substitutions from substitutions of constant length.

\subsection{Examples}
\label{ss:examples}

Let us conclude with a series of examples chosen to illustrate different aspects of our results. All of them are primitive and aperiodic (aperiodicity can be checked using \cite[Exercise~5.15]{Almeida2020a}, for instance). We use, without further mention, the fact that in these cases, every return substitution defines an $\omega$-presentation of the Schützenberger group (\theo\ref{t:return}). Return substitutions were computed using a Python implementation of an algorithm described in \cite[p.205]{Durand2012}. In every example, we give also the relevant reciprocal characteristic polynomials, and (save for \exa\ref{e:almeida}) the polynomials $\xi_1$, $\xi_2$ of \prop\ref{p:cyclo} and the integer $m_\varphi$ of \prop\ref{p:dimcyclo}. We then proceed, using \theo\ref{t:nilquotient}, to describe the pronilpotent quotients of the Schützenberger group, and we draw conclusions regarding its freeness using our various tests (\S\S\ref{ss:freeness}, \ref{ss:schutz}). 

In what follows, we use the term \emph{cyclic} as a synonym for 1-generated. We also recall the notation $\zz_p$, denoting for a prime $p$ the additive group of the $p$-adic integers. For a set of primes $\pi$, we write $\freepro_{\nil,\pi}$ instead of $\freepro_{\var{G}_{nil,\pi}}$ (the definition of $\var{G}_{\nil,\pi}$ may be recalled at the beginning of \S\ref{ss:freeness}). Our first example is a good contender for the title of ``most studied substitution''.

\begin{example}\label{e:morse}
    The \emph{Thue--Morse substitution} is the binary substitution $\tau$ defined by
    \begin{equation*}
        \tau\from\left\{\begin{array}{lll}
            0 & \mapsto & 01\\
            1 & \mapsto & 10.
        \end{array}\right.
    \end{equation*}
    Since it has constant length, the group $G(\tau)$ is not free (\theo\ref{t:uniform}). The reciprocal characteristic polynomial of $\tau$ is $-2x+1$, so the weak relative freeness test stated in \prop\ref{p:weaktest} is inconclusive. Here is the return substitution corresponding to the connection $(0,1)$ of $\tau$, which has order 2:
    \begin{equation*}
        \retsubs{\tau}{0,1}\from\left\{
            \begin{array}{rlll}
                0 & \mapsto & 0123\\
                1 & \mapsto & 013\\
                2 & \mapsto & 02123\\
                3 & \mapsto & 0213.
            \end{array}\right.
    \end{equation*}
    The reciprocal characteristic polynomial of $\retsubs{\tau}{0,1}$, together with the polynomials $\xi_1$, $\xi_2$ and the integer $m_\tau$, are as follows:
    \begin{equation*}
        \poly{\retsubs{\tau}{0,1}}^\rev(x) = (4x-1)(x-1),\quad \xi_1(x)=x-1,\quad \xi_2(x)=1,\quad m_\tau=1.
    \end{equation*}
    Hence, we may apply \coro\ref{c:reltest} with $p=2$ to conclude that $G(\tau)$ is not relatively free, thus recovering \cite[\theo7.6]{Almeida2013}. (We note that the proof given in \cite{Almeida2013} is, in some sense, similar to ours: it relies on variations in the dimensions of the maximal pro-$\var{Ab}_p$ quotients of $G(\tau)$ to reach a contradiction, much like what we do in \prop\ref{p:relfree}.) Letting $\pi$ be the set of all odd primes, we deduce from \theo\ref{t:nilquotient} that
    \begin{equation*}
        \mqf{\nil}(G(\varphi))\isom \zz_2\times\freepro_{\nil,\pi}(2).
    \end{equation*}
    A pronilpotent group is a quotient of $G(\tau)$ if and only if its 2-Sylow subgroup is cyclic and all other Sylow subgroups are 2-generated.
\end{example}

Next, we give an example of a substitution whose Schützenberger group has a cyclic maximal pronilpotent quotient. It also features a negative value for $m_\varphi$.
\begin{example}\label{e:negative}
    Consider the following ternary substitution:
    \begin{equation*}
        \varphi\from\left\{
            \begin{array}{lll}
                0 & \mapsto & 120\\
                1 & \mapsto & 121\\
                2 & \mapsto & 200.
            \end{array}
        \right.
    \end{equation*}
    Since it has constant length, its Schützenberger group is not free (\theo\ref{t:uniform}). Its reciprocal characteristic polyomial is equal to $(3x-1)(x-1)$. The pair $(0,1)$ is a connection of $\varphi$ of order 1, and the corresponding return substitution is the binary substitution
    \begin{equation*}
        \retsubs{\varphi}{0,1}\from\left\{
            \begin{array}{lll}
                0 & \mapsto & 0011\\
                1 & \mapsto & 01.
            \end{array}
        \right.
    \end{equation*}
    We give below its reciprocal characteristic polynomial, the two polynomials $\xi_1$ and $\xi_2$ and the integer $m_\varphi$:
    \begin{equation*}
        \poly{\retsubs{\varphi}{0,1}}^\rev(x)=3x-1,\quad \xi_1(x)=1,\quad \xi_2(x) = x-1,\quad m_\varphi=-1.
    \end{equation*}
    Let $\pi$ be the set of all primes distinct from 3. Following \prop\ref{p:relfree}, the maximal pronilpotent quotient of $G(\varphi)$ is free of rank 1 with respect to $\var{G}_{\nil,\pi}$. Accordingly, a pronilpotent group is a quotient of $G(\varphi)$ if and only if it is cyclic and its 3-Sylow subgroup is trivial. 
\end{example}

Next is a substitution for which the weak freeness test of \prop\ref{p:weaktest} is conclusive.
\begin{example}\label{e:weaktest}
    Consider the binary substitution
    \begin{equation*}
        \varphi\from\left\{
            \begin{array}{lll}
                0 & \mapsto & 1001\\
                1 & \mapsto & 000.
            \end{array}
        \right.
    \end{equation*}
    It satisfies $\poly{\varphi}^\rev(x) = -6x^2 - 2x + 1$, so its Schützenberger group is not relatively free (apply \prop\ref{p:weaktest} with $p_1=2$, $p_2=3$). The connection $(0,0)$ of $\varphi$, which has order 2, gives the return substitution
    \begin{equation*}
        \retsubs{\varphi}{0,0}\from\left\{
            \begin{array}{lll}
                0 & \mapsto & 0012100\\
                1 & \mapsto & 0012101221012100\\
                2 & \mapsto & 0012101222221012100.
            \end{array}
        \right.
    \end{equation*}
    We give below its reciprocal characteristic polynomial, the polynomials $\xi_1$, $\xi_2$ and the integer $m_\varphi$.
    \begin{equation*}
        \poly{\retsubs{\varphi}{0,0}}^\rev(x)=-(x-1)(36x^2-16x+1),\quad \xi_1(x)=x-1,\quad \xi_2(x)=1,\quad m_\varphi=1.
    \end{equation*}
    If $\pi$ is the set of all primes distinct from 2 and 3, then \theo\ref{t:nilquotient} yields
    \begin{equation*}
        \mqf{\nil}(G(\varphi)) \isom \zz_2\times\freepro_3(2)\times\freepro_{\nil,\pi}(3).
    \end{equation*}
    Consequently, a pronilpotent group is a quotient of $G(\varphi)$ if and only if its 2-Sylow subgroup is cyclic, its 3-Sylow subgroup is 2-generated, and all other Sylow subgroups are 3-generated.
\end{example}

We gave, in \S\ref{ss:return}, an example of a substitution on a quaternary alphabet whose return substitutions are very large. Let us revisit this example.
\begin{example}\label{e:tedious2}
    Recall the substitution $\varphi$ of \exa\ref{e:tedious1},
    \begin{equation*}
        \varphi\from\left\{
            \begin{array}{llll}
                0 & \mapsto & 12\\
                1 & \mapsto & 22\\
                2 & \mapsto & 33\\
                3 & \mapsto & 00.
            \end{array}
        \right.
    \end{equation*}
    Because it has constant length, its Schützenberger group is not free (\theo\ref{t:uniform}). We find that its reciprocal characteristic polynomial is $-(2x-1)(4x^3+4x^2+2x+1)$. Its return substitutions are too big to be represented here, but for the purpose of understanding the pronilpotent quotients of $G(\varphi)$, we only need the reciprocal characteristic polynomial of any return substitution. For instance, for the connection $(2,3)$ of $\varphi$, according to our computations,
    \begin{equation*}
        \poly{\retsubs{\varphi}{2,3}}^\rev(x) = (x-1)^6(2^{12}\,x-1)(2^{26}\,x^3 - (2^{16}\cdot 11)\,x^{2} - (2^{8}\cdot 5)\,x - 1),
    \end{equation*}
    so we have
    \begin{equation*}
        \xi_1(x) = (x-1)^6,\quad \xi_2(x)=1,\quad m_\varphi=6.
    \end{equation*}
    Applying \coro\ref{c:reltest} with $p=2$, we conclude that $G(\varphi)$ is not relatively free. Moreover, we can apply \theo\ref{t:nilquotient} to deduce the following, where $\pi$ is the set of all odd primes:
    \begin{equation*}
        \mqf{\nil}(G(\varphi))\isom \freepro_2(6)\times\freepro_{\nil,\pi}(10).
    \end{equation*}
    A pronilpotent group is a quotient of $G(\varphi)$ if and only if its $2$-Sylow component is 6-generated and all the other components are 10-generated.
\end{example}

Recall, from \rem\ref{r:cyclo}, that the polynomials $\xi_1$, $\xi_2$ of \prop\ref{p:cyclo} do not vary between connections sharing the same middle letters. Our next example shows that this is not true between arbitrary connections.
\begin{example}\label{e:cyclo}
    Consider the ternary substitution
    \begin{equation*}
        \varphi\from\left\{
            \begin{array}{lll}
                0 & \mapsto & 010\\
                1 & \mapsto & 21\\
                2 & \mapsto & 102.
            \end{array}
        \right.
    \end{equation*}
    It is unimodular, so its Schützenberger group has a free pronilpotent maximal quotient. Its reciprocal characteristic polyomial is $-(x-1)(x^2-3x+1)$. Consider the connections $(1,0)$ and $(0,1)$: they have respective order 1 and 2, and the corresponding return substitutions are
    \begin{equation*}
        \retsubs{\varphi}{1,0}\from\left\{
            \begin{array}{lll}
                0 & \mapsto & 01\\
                1 & \mapsto & 002\\
                2 & \mapsto & 0012,
            \end{array}
        \right.\quad
        \retsubs{\varphi}{0,1}\from\left\{
            \begin{array}{lll}
                0 & \mapsto & 011202312\\
                1 & \mapsto & 0112312\\
                2 & \mapsto & 012\\
                3 & \mapsto & 0112311202312.
            \end{array}
        \right.
    \end{equation*}
    With the connection $(1,0)$, we obtain the following values for the reciprocal characteristic polynomial, and the polynomials $\xi_1$, $\xi_2$:
    \begin{equation*}
        \poly{\retsubs{\varphi}{1,0}}^\rev(x)=(x+1)(x^2-3x+1),\quad \xi_1(x)=x+1,\quad \xi_2(x)=x-1,
    \end{equation*}
    while, with the connection $(0,1)$, we get instead
    \begin{equation*}
        \poly{\retsubs{\varphi}{0,1}}^\rev(x)=-(x-1)(x^2-7x+1),\quad \xi_1(x)=1=\xi_2(x).
    \end{equation*}
    In accordance with \coro\ref{c:deg}, both connections give the value $m_\varphi=0$. The return substitutions have pseudodeterminant $\pm1$, hence the maximal pronilpotent quotient of $G(\varphi)$ is a free pronilpotent group of rank 3 by \prop\ref{p:relfree}. 
\end{example}

We finish with an infinite family of examples determined by two parameters $k$ and $l$. One member of this family (the case $k=1$, $l=3$) was previously studied in early work of Almeida about maximal subgroups of free profinite monoids. To the best of our knowledge, it stands as the first published example of a non-free maximal subgroup of a free profinite monoid \cite[\exa7.2]{Almeida2007}.
\begin{example}\label{e:almeida}
    Fix $k,l\geq 0$ and let $\varphi$ be the binary substitution
    \begin{equation*}
        \varphi\from\left\{
            \begin{array}{lll}
                0 & \mapsto & 0^k1 \\
                1 & \mapsto & 0^l1.
            \end{array}
        \right.
    \end{equation*}
    Provided $l\geq 1$, it is primitive. We claim that it is aperiodic if and only if $k\neq l$. Indeed, suppose that $\varphi$ is periodic, and assume first that $k\geq l$. Let $w$ be a period of $\lang{\varphi}$, by which we mean that every word $x\in\lang{\varphi}$ is a factor of some power $w^n$, and $w$ is minimal for this property. By \cite[Exercise~5.15]{Almeida2020a}, we may in fact assume that $w$ is  is a prefix of $0^k1$, and clearly it cannot be a proper prefix; hence, we have $w=0^k1$. But $\lang{\varphi}$ also contains $10^l1$, and this can only be the case if $l=k$. The case $l\geq k$ is analogous. From now on, we assume $l\geq 1$ and $k\neq l$. 

    Next, we observe that $\varphi$ is proper, hence it defines an $\omega$-presentation of its own Schützenberger group (see \rem\ref{r:proper}). The reciprocal characteristic polynomial of $\varphi$ is given by: $\poly{\varphi}^\rev(x) = (k-l)x^2-(k+1)x+1$. By \prop\ref{p:relfree}, the maximal pronilpotent quotient of $G(\varphi)$ is free pronilpotent of rank 2 whenever $|k-l|=1$. (In fact, in that case, it is not hard to see that $\varphi$ induces an automorphism of the free group of rank 2. Such substitutions are well known to be Sturmian \cite[\coro9.2.7]{Fogg2002}, and the Schützenberger group of every Sturmian substitution must be a free profinite group of rank 2 \cite[\coro6.1]{Almeida2007}.) 

    On the other hand, when $|k-l|>1$, \coro\ref{c:abstest} implies that $G(\varphi)$ is not free. Moreover, when there is a prime that divides $k-l$ but not $k+1$, we conclude from \coro\ref{c:reltest} that the Schützenberger group is not relatively free. Let $\pi_1$ be the set of all primes that do not divide $k-l$ and $\pi_2$ be the (finite) set of all primes that divide $k-l$ but not $k+1$. We deduce from \theo\ref{t:nilquotient} that 
    \begin{equation*}
        \mqf{\nil}(G(\varphi))\isom \left(\prod_{p\in\pi_2}\zz_p\right)\times\freepro_{\nil,\pi_1}(2).
    \end{equation*}
    In particular, a pronilpotent group is a quotient of $G(\varphi)$ if and only if for every prime $p$, its $p$-Sylow component is: 2-generated if $p\in\pi_1$; cyclic if $p\in\pi_2$; trivial if $p$ divides $\gcd(k-l,k+1)$.
\end{example}
Using other means, the group $G(\varphi)$ above was shown not to be relatively free in the case $k=1$ and $l=3$ \cite[\theo7.2]{Almeida2013}, but this case is not covered by \coro\ref{c:reltest}. In fact, in light of our results, the pronilpotent quotients alone do not contain enough information about $G(\varphi)$ to reach this conclusion. Indeed, in that case, $\poly{\varphi}^\rev(x)=-2x^2-2x+1$ and \theo\ref{t:nilquotient} implies $\mqf{\nil}(G(\varphi))\isom\freepro_{\nil,\pi}(2)$, where $\pi$ is the set of all odd primes.

\section*{Acknowledgements}

This work was partially supported by the Centre for Mathematics of the University of Coimbra (UIDB/00324/2021, funded by the Portuguese Government through FCT/MCTES) and the Centre for Mathematics of the University of Porto (UIDB/00144/2020, funded by the Portuguese Government through FCT/MCTES). The author also benefited from a scholarship (PD/BD/150350/2019, funded by the Portuguese Government through FCT/MCTES).

My heartfelt thanks to Alfredo Costa, whose guidance and advice were invaluable throughout the long process of preparing this paper. Many thanks also to Jorge Almeida for his very helpful comments.

\bibliographystyle{abbrv} 
\bibliography{nilpotent_quotients}

\begin{thebibliography}{10}

\bibitem{Almeida2003}
J.~Almeida.
\newblock Profinite structures and dynamics.
\newblock {\em CIM Bull.}, 14:8--18, 2003.

\bibitem{Almeida2004}
J.~Almeida.
\newblock Symbolic dynamics in free profinite semigroups.
\newblock In {\em Algebraic Systems, Formal Languages and, Conventional and
  Unconventional Computation Theory}, volume 1366, pages 1--12. RIMS Kokyuroku,
  2004.

\bibitem{Almeida2007}
J.~Almeida.
\newblock Profinite groups associated with weakly primitive substitutions.
\newblock {\em J. Math. Sci.}, 144(2):3881--3903, 2007.
\newblock Translated from \emph{Fundam. Prikl. Mat.}, 11(3):13--48, 2005.

\bibitem{Almeida2013}
J.~Almeida and A.~Costa.
\newblock Presentations of {S}ch\"utzenberger groups of minimal subshifts.
\newblock {\em Israel J. Math.}, 196(1):1--31, 2013.

\bibitem{Almeida2020a}
J.~Almeida, A.~Costa, R.~Kyriakoglou, and D.~Perrin.
\newblock {\em Profinite Semigroups {a}nd Symbolic Dynamics}.
\newblock Springer International Publishing, 2020.

\bibitem{Almeida2006}
J.~Almeida and M.~V. Volkov.
\newblock Subword complexity of profinite words and subgroups of free profinite
  semigroups.
\newblock {\em Int. J. Algebra Comput.}, 16(02):221--258, 2006.

\bibitem{Costa2006}
A.~Costa.
\newblock Conjugacy invariants of subshifts: An approach from profinite
  semigroup theory.
\newblock {\em Int. J. Algebra Comput.}, 16(4):629--655, 2006.

\bibitem{Costa2021}
A.~Costa and B.~Steinberg.
\newblock The {K}aroubi envelope of the mirage of a subshift.
\newblock {\em Commun. Algebra}, 49:4820--4856, 2021.

\bibitem{Durand1998}
F.~Durand.
\newblock A characterization of substitutive sequences using return words.
\newblock {\em Discrete Math.}, 179(1-3):89--101, 1998.

\bibitem{Durand1998a}
F.~Durand.
\newblock A generalization of {C}obham's theorem.
\newblock {\em Theory Comput. Syst.}, 31(2):169--185, 1998.

\bibitem{Durand2012}
F.~Durand.
\newblock {HD0L}-{$\omega$}-equivalence and periodicity problems in the
  primitive case (to the memory of {G}. {R}auzy).
\newblock {\em Unif. Distrib. Theory}, 7(1):199--215, 2012.

\bibitem{Durand1999}
F.~Durand, B.~Host, and C.~Skau.
\newblock Substitutional dynamical systems, {B}ratteli diagrams and dimension
  groups.
\newblock {\em Ergodic Theory Dynam. Systems}, 19(4):953--993, 1999.

\bibitem{Fogg2002}
N.~P. Fogg.
\newblock {\em Substitutions in Dynamics, Arithmetics and Combinatorics}.
\newblock Springer Berlin Heidelberg, 2002.

\bibitem{Fried2008}
M.~D. Fried and M.~Jarden.
\newblock {\em Field Arithmetic}.
\newblock Springer Berlin Heidelberg, 2008.

\bibitem{GAP2020}
The GAP~Group.
\newblock {\em {GAP -- Groups, Algorithms, and Programming, Version 4.11.0}}.

\bibitem{GouletOuellet2021}
H.~Goulet-Ouellet.
\newblock Freeness of {Sch{\"{u}}tzenberger} groups of primitive substitutions.
\newblock {\em {DMUC} preprints}, 21(35), 2021.
\newblock Preprint available at \url{www.mat.uc.pt/preprints/2021.html}. To
  appear in \emph{Int. J. Algebra Comput.}

\bibitem{Hunter1983}
R.~P. Hunter.
\newblock Some remarks on subgroups defined by the {Bohr} compactification.
\newblock {\em Semigr. Forum}, 26(1):125--137, 1983.

\bibitem{Lang2002}
S.~Lang.
\newblock {\em Algebra}.
\newblock Springer New York, 2002.

\bibitem{Lubotzky2001}
A.~Lubotzky.
\newblock Pro-finite presentations.
\newblock {\em J. Algebra}, 242(2):672--690, 2001.

\bibitem{MacLane1971}
S.~Mac~Lane.
\newblock {\em Categories for the Working Mathematician}.
\newblock Springer New York, 1971.

\bibitem{Rhodes2008}
J.~Rhodes and B.~Steinberg.
\newblock Closed subgroups of free profinite monoids are projective profinite
  groups.
\newblock {\em Bull. Lond. Math. Soc.}, 40(3):375--383, 2008.

\bibitem{Ribes2010a}
L.~Ribes and P.~Zalesskii.
\newblock {\em Profinite Groups}.
\newblock Springer Berlin Heidelberg, second edition, 2010.

\bibitem{Sage2020}
The Sage Developers.
\newblock {\em {SageMath, the Sage Mathematics Software System, Version 9.2}},
  2020.

\bibitem{Weintraub2019}
S.~H. Weintraub.
\newblock {\em Linear Algebra for the Young Mathematician}.
\newblock American Mathematical Society, 2019.

\end{thebibliography}

\end{document}